\documentclass[a4paper,10pt]{article}

\usepackage{amsmath,amssymb,amsthm,graphicx,
	subfig,
verbatim,cite,
} 

\usepackage{geometry}
\usepackage{tikz}
\usetikzlibrary{shapes.geometric,calc,decorations.markings,math}

\tikzstyle{nodo}=[circle,draw,fill,inner sep=0pt,minimum size=%
1.5mm]
\tikzstyle{infinito}=[circle,inner sep=0pt,minimum size=0mm]

\newcommand\R{{\mathbb R}}
\newcommand\Z{{\mathbb Z}}
\newcommand\N{{\mathbb N}}

\newcommand\LL{\mathcal L}

\newcommand\Sf{\mathcal S}

\newcommand\BB{{\mathcal B}}
\newcommand\f{\frac}

\newcommand\SL{{\mathcal S}_\lambda}

\newcommand\G{\mathcal G}

\newcommand\HH{\mathcal H}
\newcommand\NN{\mathcal N}
\newcommand\RR{\mathcal R}

\newcommand\vv{\textsc{v}}

\newcommand\eps{\varepsilon}

\newcommand\NL{{\mathcal N}_\lambda}

\newcommand{\Gt}{{\widetilde\G_N}}

\theoremstyle{definition}

\theoremstyle{plain}
\newtheorem{theorem}{Theorem}[section]
\newtheorem{proposition}[theorem]{Proposition}
\newtheorem{lemma}[theorem]{Lemma}
\newtheorem{corollary}[theorem]{Corollary}

\newcounter{ass}
\theoremstyle{remark}
\newtheorem{remark}[theorem]{Remark}
\newtheorem*{remark*}{Remark}

\theoremstyle{definition}
\newtheorem{definition}[theorem]{Definition}

 \numberwithin{equation}{section}

\date{}

\title{On the notion of ground state for nonlinear Schr\"odinger equations on metric graphs }

\author{Colette De Coster$^1$, Simone Dovetta$^2$, Damien Galant$^{1,3}$, Enrico Serra$^2$ 
	\\ \ \\{\small$^1$ Univ. Polytechnique Hauts-de-France, INSA Hauts-de-France, CERAMATHS - Laboratoire de}
	\\ {\small Mat\'eriaux C\'eramique et de Math\'ematiques, F-59313 Valenciennes, France}
	 \\{\small$^2$Dipartimento di Scienze
		Matematiche ``G.L. Lagrange'', Politecnico di Torino } \\ {\small
		Corso Duca degli Abruzzi, 24, 10129 Torino, Italy} \\ 
		{\small$^3$ {\small  F.R.S.-FNRS and UMONS - Universit\'e de Mons, Mons, Belgium}
		}}

\begin{document}

\maketitle

\begin{abstract} 
We compare ground states for the nonlinear Schr\"odinger equation on metric graphs, defined as global minimizers of the action functional constrained on the Nehari manifold, and least action solutions, namely minimizers of the action among all solutions to the equation.
In principle, four alternative cases may take place: ground states do exist (thus coinciding with least action solutions); ground states do not exist while least action solutions do; both ground states and least action solutions do not exist and the levels of the two minimizing problems coincide; both ground states and least action solutions do not exist and the levels of the two minimizing problems are different. We show that in the context of metric graphs all four alternatives do occur. This is accomplished by a careful analysis of doubly constrained variational problems.
As a by-product, we obtain new multiplicity results for positive solutions on a wide class of noncompact metric graphs.
 \end{abstract}

\smallskip


\section{Introduction}

Nonlinear Schr\"odinger equations on metric graphs have attracted the interest of a large -- and increasing -- number of researchers in the last few years. As the literature on the subject witnessed a massive growth, we refrain from overviewing it here, redirecting e.g. to \cite{ABD,ACFN,AP,BMP,BDL20,BDL21,BD22,BD21,DT22,CJS,
BCJS,BCJS2, G,GKP,KMPX,NP,PS22,PSV}  for some of the most recent developments and to the reviews \cite{ABR,KNP} for more comprehensive discussions. Within the whole theory, prominent efforts have been devoted to the analysis of existence of positive standing wave solutions, with a particular focus on {\em ground states}. 

The notion of ground state however, albeit often suggested unequivocally by the specific problem under study, is by no means 
univocally defined. This aspect is of course not specific to Schr\"odinger equations on metric graphs, but is a general feature appearing in the study of various types of  scalar field equations on a variety of domains, from open subsets of $\R^N$ to Riemannian manifolds.

To describe it more concretely, we consider a metric graph $\G$ and the NLS equation
\begin{equation}
\label{NLS0}
u'' + |u|^{p-2}u = \lambda u \quad\text{on } \G
\end{equation}
where $\lambda >0$ and $p>2$. As usual, it is required that \eqref{NLS0} be satisfied pointwise on every edge of $\G$, while additional matching conditions have to be imposed at the vertices of the structure. In this paper, equation \eqref{NLS0} is coupled with the so--called natural, or Kirchhoff, boundary conditions, prescribing that on every vertex $\vv$ of $\G$ the sum of the outgoing derivatives of $u$ along every edge incident at $\vv$ is zero.
Thus, defining a coordinate $x_e \in (0, |e|)$ on every edge $e$ of $\G$ (where $|e|$ is the length of edge $e$), the problem we are addressing reads
\begin{equation}
\label{NLS}
\begin{cases} u'' + |u|^{p-2}u = \lambda u & \text{on every edge $e$ of $\G$} \\[1ex]
u \text{ is continuous} & \text{at every vertex $\vv$ of $\G$}\\[1ex]
\displaystyle\sum_{e\succ \vv}  \frac{du}{dx_e}(\vv)=0 & \text{at every vertex $\vv$ of $\G$,}
\end{cases}
\end{equation}
where $\lambda>0$, $p>2$, the symbol $e\succ \vv$ means that the sum is extended to all edges emanating from $\vv$
and where $\frac{du}{dx_e}(\vv)$ is the outward derivative of $u$ at $\vv$. 
In this framework, the solutions to \eqref{NLS} can be characterized variationally as the critical points of the standard action functional $J_{\lambda} : H^1(\G) \to \R$
\begin{equation}
\label{Jlambda}
J_{\lambda}(u) := \frac{1}{2} \|u'\|_{L^2(\G)}^2 + \frac{\lambda}{2} \|u\|_{L^2(\G)}^2 - \frac{1}{p} \|u\|_{L^p(\G)}^p,
\end{equation}
where 
\begin{equation*}
H^1(\G) = \Bigl\{ u : \G \rightarrow \R \mid u \text{ is continuous and } u, u' \in L^2(\G) \Bigr\}.
\end{equation*}
In the search for ground states, as $J_{\lambda}$ is not bounded from below, one may impose extra constraints to recover boundedness and make a minimization procedure meaningful.
For instance one could restrict $J_{\lambda}$ to the unit sphere of $L^p(\G)$ or to the Nehari manifold associated with $J_{\lambda}$
(this second approach has the advantage that the nonlinearity need not be homogeneous).
Both procedures make $J_{\lambda}$ bounded from below on these sets and one then defines ground states as the functions that achieve the infimum of $J_{\lambda}$ on the constraint.
Specifically, for our problem, the Nehari manifold is the set
\begin{align}
\label{Nehari}
\mathcal{N}_{\lambda}(\G) :=& \, \left\{ u \in H^1(\G)\setminus\{0\} \mid J_{\lambda}'(u)u= 0\right\}\nonumber\\
=& \, \left\{ u \in H^1(\G)\setminus\{0\} \mid  \|u'\|_{L^2(\G)}^2 + \lambda \|u\|_{L^2(\G)}^2 = \|u\|_{L^p(\G)}^p\right\}.
\end{align}
As is well known, the Nehari manifold contains all nonzero critical points of $J_{\lambda}$ and is a natural constraint, in the sense that constrained critical points are in fact true critical points of $J_{\lambda}$. This leads to a first definition of ground state. Defining
\[
c_\lambda(\G) = \inf_{u \in \NL(\G) } J_{\lambda}(u),
\]
it is customary to set the following
\begin{definition} 
\label{gsadef}
An {\em action ground state}  for \eqref{NLS} is a function $u \in \NL(\G)$ such that
\begin{equation}
\label{gsa}
J_{\lambda}(u) = c_\lambda(\G).
\end{equation}
\end{definition}

Clearly any action ground state is a constant sign solution to \eqref{NLS}, often referred to also as a ``solution of minimal action'',
though we adopt here a different terminology (see Definition \ref{solmadef}).
In applications, action ground states play a prominent role for various reasons, which fully justifies the preceding definition.
In practice, however, one is very often confronted with the following inconvenience:
there need not exist any function $u$ in $\NL(\G)$ satisfying \eqref{gsa}
(a frequent fact in many noncompact settings).
Thus, the existence of a ground state is not guaranteed in general.

To overcome this obstacle, sometimes one assumes a different definition of ground state. Let
\[
\SL(\G) =   \, \{ u \in H^1(\G) \setminus\{0\} \mid  \ u \text { solves } \eqref{NLS} \}
\]
be the set of nonzero solutions to \eqref{NLS} and set
\begin{equation}
\label{solgs}
\sigma_\lambda(\G) = \inf_{u \in \SL(\G) } J_{\lambda}(u).
\end{equation}
Although it is common use to call ground states too the functions described in the following definition, we prefer to label them with a different name, to avoid misunderstandings.

\begin{definition} 
\label{solmadef}
A {\em least action solution} for \eqref{NLS} is a function $u \in \SL(\G)$ such that
\begin{equation}
\label{solma}
J_{\lambda}(u) = \sigma_\lambda(\G).
\end{equation}
\end{definition}

The two points of view motivating the above definitions are clearly different. In the former case one fixes the attention on the level $c_\lambda(\G)$, and tries to prove that it is attained, obtaining in this way a solution that has minimal action among all {\em functions} in $\NL(\G)$. In the latter the aim is to ascertain if, {\em among solutions} of \eqref{NLS}, there is one of least action; in this case there might be plenty of functions with lesser action (and none of them will solve problem \eqref{NLS} if $c_\lambda(\G)$ is not attained).

Although $\SL(\G)$ is a much smaller set than $\NL(\G)$ it is not at all clear, in general, that there exist any functions in $\SL(\G)$ that achieve $\sigma_\lambda(\G)$. Thus, even with the second definition, the existence of ``ground states'' is not guaranteed {\em a priori}.

The aim of this paper is to analyze the relations among $c_\lambda(\G)$ and $\sigma_\lambda(\G)$ and in particular to investigate if all the theoretical cases can actually take place. To make this point clear we start by observing that, trivially,
\begin{equation}
\label{csigma}
c_\lambda(\G) \le \sigma_\lambda(\G),
\end{equation}
without any further assumption. Furthermore, if $c_\lambda(\G)$ is attained by some function $u \in \NL(\G)$,
then $u \in \SL(\G)$, the equality $c_\lambda(\G) = \sigma_\lambda(\G)$ holds, and $\sigma_\lambda(\G)$ is attained too.
In view of these preliminary considerations, the possibilities to consider are exactly the following four:
\begin{itemize}
\item[A1)] $c_\lambda(\G) = \sigma_\lambda(\G)$, and they are attained;
\item[A2)] $c_\lambda(\G) = \sigma_\lambda(\G)$, and they are not attained;
\item[B1)] $c_\lambda(\G) < \sigma_\lambda(\G)$, $\sigma_\lambda(\G)$ is attained and $c_\lambda(\G)$ is not;
\item[B2)] $c_\lambda(\G) < \sigma_\lambda(\G)$, and neither is attained.
\end{itemize}
To prove or disprove the actual occurrence of these situations then amounts, for each A1--B2, to produce an example of a graph $\G$ where the behavior is exactly the one prescribed by the alternative in question.

We wish to make very clear that the previous discussion is completely independent of the domain on which the NLS equation is set.
For instance, replacing $\G$ by an open subset $\Omega$ of $\R^N$ (and the second derivative by the Laplacian,
and so on, e.g. $H^1(\G)$ by $H^1(\Omega)$ or  $H^1_0(\Omega)$,....), every issue described so far can be stated in the new context without any modification except notation.
In particular, the four alternatives listed above remain and it would be extremely interesting to understand if they really occur.
While this can be easily achieved in some of the alternatives (A1, for instance), some of them appear rather difficult to deal with when considering the NLS equation on subsets of $\R^N$ and are currently out of reach.

Our present contribution aims at shedding some light on this problem starting from the context of metric graphs, where one can profitably use the advantage of the local one--dimensional nature of the ambient space to obtain sharper results. Even in this setting, though, the constructions we will provide are quite involved, very specific to the metric graph environment and do not seem to be extendable to other frameworks.

Stated in a compact form, our main result  is the following. In a nutshell, as far as action ground states are concerned, ``anything can happen'' on metric graphs.

\begin{theorem} 
\label{main}
For every $p>2$, every $\lambda >0$ and every choice of an alternative among A1, A2, B1, B2, there exists a metric graph $\G$ where that alternative takes place.
\end{theorem}
As the proof of Theorem \ref{main} will be carried out case by case, let us briefly comment here on A1--B2, postponing the technical issues to subsequent sections.

Case A1 corresponds to problems where the infimum $c_\lambda(\G)$ is achieved. This is what one normally tries to obtain in the existence results and it is the case for all compact graphs and for some noncompact ones  (see e.g. \cite{AST1,AST2,AST3,D,DGMP,DT19,KS,NPS,NP,T}, works that deal with still a slightly different notion of ground state, but whose techniques adapt easily  to the present setting). The theory in this framework is rather well developed, not only on metric graphs of course, and there is not much to add.

In case A2 the graph is necessarily noncompact.  There are plenty of examples where it is known that $c_\lambda(\G)$ is not attained, and this is due to topological or metric obstructions on the graph that have been widely described in the literature (see again e.g. \cite{AST1,AST2,AST3,DT19}). 
Nonetheless, this leaves open the question of the existence of a least action solution in the sense of Definition \ref{solmadef}.
As far as we know, examples of this kind,  where $c_\lambda(\G)$ coincides with $\sigma_\lambda(\G)$ but the level is not achieved, have never been described before. The construction of a graph with this property will be the object of Section \ref{subsec:A2} (Theorem \ref{Gc}) and is one of the principal proofs of the paper.

Due to existing results, it is easy to produce an example where alternative B1 occurs. We will briefly describe it anyway for completeness in Section \ref{subsec:B1}, since it has never been considered under this perspective.

Case B2 is the hardest one and will be treated in Section \ref{subsec:B2} (Theorem \ref{prop:largeN}). It is well known that, typically,  the lack of a function attaining $c_\lambda(\G)$ is due to the presence of a ``problem at infinity'' that attracts nonconvergent minimizing sequences. This is a standard phenomenon in problems with lack of compactness and is essentially what takes places in cases A2 and B1. The main novelty in B2, which is what makes this case rather delicate, is that the infimum over solutions is not attained due to the presence of a {\em second} problem at infinity, at level $\sigma_\lambda(\G) > c_\lambda(\G)$. By this we do not mean a problem at infinity with loss of compactness at different levels, but rather the presence of two {\em distinct} problems: the first, as we said, attracting nonconvergent minimizing sequences for $J_{\lambda}$ and preventing at the same time the existence of solutions with action arbitrarily close to such infimum level; the second attracting nonconvergent sequences of solutions of lower and lower level. In fact, the gap between $c_\lambda(\G)$ and $\sigma_\lambda(\G)$ reflects the coexistence of this pair of necessarily different problems at infinity. This seems to be a new phenomenon and it is what makes quite involved the construction of a graph exhibiting it.

As a byproduct, we notice that in cases A2 and B2, the fact that $\sigma_\lambda(\G)$ is not attained immediately implies the existence of infinitely many positive solutions for \eqref{NLS}, with accumulating levels, which we believe to be a remarkable fact. 
\medskip

From the technical point of view, the proofs of the aforementioned results exploit deeply the role of both the topology and the metric of graphs to determine existence/non--existence of solutions to specific variational problems. On the one hand, non--existence of ground states is obtained via by--now standard arguments in the theory of NLS on metric graphs. On the other hand, the construction of noncompact sequences of solutions in $\SL(\G)$ with specific action level is achieved, in both cases A2 and B2, through a careful analysis of doubly--constrained minimization problems in the form
\begin{equation}
\label{doubleconst}
\inf_{u \in \NL(\G)\cap X_e}J_{\lambda}(u)
\end{equation}
where
\[
X_e=\left\{u\in H^1(\G)\,:\,\|u\|_{L^\infty(\G)}=\|u\|_{L^\infty(e)}\right\}
\]
is the subset of $H^1$ functions attaining their $L^\infty$ norm on a given bounded edge $e$ of $\G$.
A rather general existence result of independent interest is derived for this kind of problems in Section \ref{sec:doppiovinc}.
Precisely, for a wide class of noncompact graphs, given $\lambda >0$ we show that the infimum in \eqref{doubleconst}
is attained whenever the length of $e$ is larger than a threshold depending only on $\lambda$ and $p$ (Theorem \ref{exlarge}).
Moreover, such a minimizer is a solution of problem \eqref{NLS} provided $e$ is sufficiently long (Theorem \ref{usol}),
the threshold depending this time also on $\displaystyle \inf_{e \in \mathbb E} |e|$, where $\mathbb E$ is the set of all edges of $\G$. This approach was originally introduced for mass--constrained critical points of the energy functional in \cite{ASTbound},
where it proved suitable to obtain multiplicity results. However, in that paper a crucial assumption is that the (prescribed) mass be sufficiently large. By scaling properties,
this is equivalent to the assumption that the length of all bounded edges is large. Here, on the contrary, it is sufficient that a {\em single} edge is long enough.
A direct consequence of these results is the existence of multiple positive solutions, each attaining its maximum on one of the edges longer than the threshold (Theorem \ref{multiple}).

To conclude this introduction, let us point out that, as is well--known, to look at solutions of \eqref{NLS} as critical points of the action functional \eqref{Jlambda}
is not the unique variational characterization at disposal. In particular, moving from the seminal papers \cite{CL,J},
in the past decade a lot of attention has been devoted, both on metric graphs and on domains in $\R^N$,
to normalized solutions, i.e. critical points of the energy functional $E:H^1(\G)\to\R$
\[
E(u):=\f12\|u'\|_{L^2(\G)}^2-\f1p\|u\|_{L^p(\G)}^p
\]
constrained to the space of functions with prescribed mass  $\mu$
\[
H_\mu^1(\G):=\left\{u\in H^1(\G)\,:\,\|u\|_{L^2(\G)}^2=\mu\right\}.
\]
In this setting, the parameter $\lambda$ appearing in \eqref{NLS} is not known a priori and pops up as a Lagrange multiplier associated to the mass constraint. It is evident that one can consider definitions of energy ground states and least energy solutions analogous to those given above when dealing with $J_{\lambda}$. Precisely, letting
\[
\widehat{c}_\mu(\G)=\inf_{u\in H_\mu^1(\G)}E(u)
\]
and 
\[
\widehat{\sigma}_\mu(\G)=\inf_{u\in\widehat{\mathcal{S}}_\mu}E(u)\,,
\]
where
\[
\widehat{\mathcal{S}}_\mu(\G):=\left\{u\in H_\mu^1(\G)\,:\, u\text{ solves }\eqref{NLS}\text{ for some }\lambda\in\R\right\}\,,
\]
one has the following mutually exclusive four alternatives (the analogue of A1--B2)
\begin{itemize}
	\item[$\widehat {\rm A1}$)] $\widehat{c}_\mu(\G) = \widehat{\sigma}_\mu(\G)$, and they are attained;
	\item[$\widehat {\rm A2}$)] $\widehat{c}_\mu(\G) = \widehat{\sigma}_\mu(\G)$, and they are not attained;
	\item[$\widehat {\rm B1}$)] $\widehat{c}_\mu(\G) < \widehat{\sigma}_\mu(\G)$, $\widehat{\sigma}_\mu(\G)$ is attained and $\widehat{c}_\mu(\G)$ is not;
	\item[$\widehat {\rm B2}$)] $\widehat{c}_\mu(\G) < \widehat{\sigma}_\mu(\G)$, and neither is attained.
\end{itemize}
As a matter of fact, the analysis we develop here for $J_{\lambda}$ can be naturally adapted to prove (with the very same constructions)
that also in the context of normalized critical points of $E$ all cases $\widehat {\rm A1}-\widehat {\rm B2}$ do actually occur. 
Note that, even though both the approaches have been widely exploited in the literature, a detailed discussion of the relation between ground states
(and more generally local minima) of $J_{\lambda}$ and $E$ was started only recently in \cite{DSTma, JL}.

\smallskip
The paper is organized as follows. Section \ref{sec:prel} recalls some preliminary results, whereas Section \ref{sec:doppiovinc} deals with doubly--constrained variational problems for the action functional in a general setting, proving the existence results in Theorems \ref{exlarge}--\ref{usol}-\ref{multiple}. Section \ref{sec:proof2} is devoted to the proof of Theorem \ref{main}.

\smallskip
\noindent\textbf{Notation.} In what follows, we will write e.g. $\|u\|_p,\dots$, in place of $\|u\|_{L^p(\G)},\dots$, whenever possible. When needed, the full notation will be used to indicate explicitly the domain of integration.

\section{Preliminaries}
\label{sec:prel}

In this paper we use a number of properties and results that have been established (mostly) in the recent literature. For the ease of the reader we collect them in the present section, referring each time to the original papers where proofs can be found.

We assume that the reader is familiar with the concept of metric graph. However, we make precise that in this paper we consider
metric graphs $\G = ({\mathbb V},{\mathbb E})$ satisfying the following 
 definition.

\begin{definition}
We denote by $\bf G$ the class of metric graphs $\G =  ({\mathbb V},{\mathbb E})$ such that
\begin{itemize}
\item $\G$ is connected and has an at most countable number of edges;
\item $\G$  has at least one unbounded edge (i.e. a  half-line);
\item $\deg(\vv) < \infty$ for every $\vv \in \mathbb{V}$, where $\deg(\vv)$ denotes the degree of the vertex $\vv$, i.e. the  number of edges emanating from it;
\item $\forall \vv\in \mathbb{V}$, $\deg(\vv) \not=2$;
\item $\inf_{e \in {\mathbb E}} |e| >0$, where $|e|$ denotes the length of $e$.
\end{itemize}
\end{definition}
There is no loss of generality in assuming that $\deg(\vv) \not=2$ since every vertex of degree two can a priori be eliminated from any metric graph, by melting the two edges incident at $\vv$ into a single edge. 
Note that every $\G \in \bf G$ is noncompact.
Further assumptions will be made when needed in the course of the paper.

In the study of the NLS equation on a metric graph in class $\bf G$, a fundamental tool is provided by the Gagliardo--Nirenberg inequalities (see  \cite{AST2})

\begin{equation}
\label{GN}
\|u\|_{q}^q\leq K\|u\|_{2}^{\frac q2+1}\|u'\|_{2}^{\frac q2-1}, \qquad K = 2^{\frac{q}{2} - 1}, 
\end{equation}
that hold for every $u\in H^1(\G)$ and every $q\geq 2$, and their $L^\infty$ version
\begin{equation}
\label{GNinfty}
\|u\|_{\infty}^2\leq 2\|u\|_{2} \|u'\|_{2}.
\end{equation}
A second fundamental tool that we will use very frequently in the next sections is provided by the rearrangement techniques of $H^1$ functions on  a generic graph $\G$, for the details of which we refer to Section 3 of \cite{AST1}. For the reader's convenience we recall here that,
given a nonnegative function $u\in H^1(\G)$, the  \emph{decreasing}  rearrangement of $u$ is the unique nonincreasing function $u^* \in H^1(0, |\G|)$ equimeasurable with $u$.
The equimeasurability property entails that 
\begin{equation}
\label{equidec}
\|u^*\|_{L^q(0, |\G|)} 
= \|u\|_{L^q(\G)}, \qquad \text{for every $q\in [1,+\infty]$}.
\end{equation}
Moreover, by the classical P\'olya--Szeg\H{o} inequality, we have
\begin{equation}
\label{PS*}
\|(u^*)'\|_{L^2(0, |\G|)} \le \|u'\|_{L^2(\G)}.
\end{equation}
Similarly,  the {\em symmetric} rearrangement $\widehat u \in H^1(-|\G|/2,|\G|/2)$ of $u$ is $\widehat u (x) = u^*(2|x|)$. By definition, $\widehat u$ is symmetric, nonincreasing on $[0, -|\G|/2)$ and equimeasurable with $u$. Furthermore, it is well known (see e.g. \cite{AST1}) that if $\# u^{-1}(t) \ge 2$ for almost every $t \in (0, \|u\|_\infty)$, then
\begin{equation}
\label{equisym}
\|\widehat u'\|_{L^2(-|\G|/2,|\G|/2)} \le \|u'\|_{L^2(\G)},
\end{equation}
where equality implies that $\#u^{-1}(t)=2$ for almost every $t \in (0, \|u\|_\infty)$.

The following result, that will be used in Section \ref{sec:proof2}, is essentially a refinement of the P\'olya--Szeg\H{o} inequality
and its proof follows combining results of \cite{Duff} and \cite{F}.

\begin{proposition}
\label{Ncontr}
Let $\G$ be any metric graph and let $u$ be a nonnegative function in $H^1(\G)$. Let $T\subset [0,+\infty[$ of positive measure. Assume that, for  some integer $K \ge 1$,
\[
\#u^{-1}(s) \ge K \qquad\text{for a.e. $s\in T$.}
\]
 Then the decreasing rearrangement  $u^*$ of $u$ satisfies
\[
\|u^*\|_{L^q((u^*)^{-1}(T))} =  \|u\|_{L^q(u^{-1}(T))}, \,\, \text{for every $q\in [1,+\infty]$}
\] 
and
\[
\|(u^*)'\|_{L^2((u^*)^{-1}(T))} \le \frac1K \|u'\|_{L^2(u^{-1}(T))}.
\]
\end{proposition}

Next, if $\G$ is a metric graph in class $\bf G$  and $\lambda >0$, we define the action  functional $J_{\lambda}\in C^1(H^1(\G),\R)$ as 
\[
J_{\lambda}(u) := \frac{1}{2} \|u'\|_2^2 + \frac{\lambda}{2} \|u\|_2^2 - \frac{1}{p} \|u\|_p^p
\]
and the Nehari manifold associated to $J_{\lambda}$,
\begin{align*}
\NN_{\lambda}(\G) :=& \, \left\{ u \in H^1(\G)\setminus\{0\} \mid J_{\lambda}'(u)u= 0\right\} \\
=& \, \left\{ u \in H^1(\G)\setminus\{0\} \mid  \|u'\|_2^2 + \lambda \|u\|_2^2 = \|u\|_p^p\right\}.
\end{align*}
It is well known that there exists a natural projection of $H^1(\G) \setminus \{ 0 \}$ on $\NN_{\lambda}$.
Defining $\pi_\lambda : H^1(\G)\setminus\{0\} \to \R$ by
\begin{equation}
\label{pi}
\pi_\lambda (u) = \left(\frac{\|u'\|_2^2 + \lambda \|u\|_2^2}{\|u\|_p^p} \right)^{\frac{1}{p-2}},
\end{equation}
we have $u \in \NN_{\lambda}(\G)$ if and only if $\pi_\lambda(u) = 1$.
We will also need the functional $L:  H^1(\G)\setminus\{0\} \to \R$ defined by
\begin{equation}
\label{L}
L(u)=\frac{\|u\|_p^p-\|u'\|_2^2}{\|u\|_2^2}.
\end{equation}
If $u$ solves problem \eqref{NLS}, then $L(u) = \lambda$ and, more generally, $L(u) = \lambda$ if and only if $u \in  \NN_{\lambda}(\G)$.

When $u \in \NN_{\lambda}(\G)$, the functional $J_{\lambda}$ takes the simple form
\begin{equation}
\label{formJ}
J_{\lambda}(u) = \kappa \|u\|_p^p = \kappa (\|u'\|_2^2 + \lambda \|u\|_2^2), \qquad \kappa = \frac12 -\frac1p,
\end{equation}
so that $J_{\lambda}$ is positive on $ \NN_{\lambda}(\G)$. Actually more can be said, and we summarize it in the next proposition.

\begin{proposition}
\label{boundedness}
There exists a constant $C>0$ depending only on $\lambda$ and $p$ such that 
\begin{equation}
\label{pbelow}
	\inf_{u \in \NN_{\lambda}(\G)} \| u \|_p \geq C>0.
\end{equation}
Moreover, if $(u_n)_n \subset \NN_\lambda(\G)$
satisfies $\displaystyle \sup_n J_{\lambda}(u_n) < \infty$,
then $(u_n)_n$ is bounded in $H^1(\G)$ and
\begin{equation*}
\inf_n \|u_n\|_2 > 0, \quad \inf_n \|u_n\|_\infty > 0.
\end{equation*}
\end{proposition}

\begin{proof}
Since $\sqrt{\|u'\|_2^2 + \lambda \|u\|_2^2}$  is equivalent to the usual $H^1(\G)$ norm,
Sobolev inequalities imply the existence of $C=C(p,\lambda) > 0$ such that,
for all $u \in \NN_{\lambda}$,
\[
\|u\|_p\le C(\|u'\|_2^2 + \lambda \|u\|_2^2)^{1/2}  = C\|u\|_p^{p/2},
\]
whence
\[
\inf_{u \in \NN_{\lambda}} \|u\|_p \ge C^{\frac2{2-p}} > 0
\]
which proves \eqref{pbelow}.

From \eqref{formJ}, as $\displaystyle \sup_n J_{\lambda}(u_n) < \infty$,  we see that $(u_n)_n$ is bounded in $H^1(\G)$,
hence in $L^2(\G)$ and $L^{\infty}(\G)$. Observing that
\[
\|u_n\|_p^p \le \|u_n\|_\infty^{p-2}\|u_n\|_2^2,
\]
we see that $ \|u_n\|_2$ and $\|u_n\|_\infty$ are also uniformly bounded away from zero.
\end{proof}

When $\G=\R$, the non-trivial solutions to \eqref{NLS} are called {\em solitons}, they are unique up to translations and sign
and they are the action ground states of $J_{\lambda}$ over 
$\NN_\lambda(\R)$ (see e.g. \cite[Proposition 3.12]{LC}).
Denoting by $\phi_\lambda$ 
the unique positive and even soliton, letting $s_1 := J_{1}(\phi_1)$ and
\begin{equation}
\label{solitonlevel}
s_\lambda := s_1 \lambda^\alpha, \qquad \alpha = \frac{p+2}{2(p-2)}, 
\end{equation}
for $\lambda > 0$, there results
\begin{equation}
\label{infR}
J_{\lambda}(\phi_\lambda)
= \inf_{u \in \NN_{\lambda}(\R)} J_{\lambda}(u)
= s_\lambda.
\end{equation}
When $\G=\R^+$, for every $\lambda >0$ there is a unique positive action ground state, given by the restriction of $\phi_\lambda$ to $\R^+$ and
\begin{equation}
\label{infR+}
\inf_{u \in \NN_\lambda(\R^+)}J_{\lambda}(u)
= J_{\lambda}\Bigl(\left.\phi_\lambda \right._{|_{\R^+}} \Bigr)
= \frac12 s_\lambda.
\end{equation}
The level $s_\lambda$ plays a fundamental role in many papers, including the present one.

\begin{lemma}
\label{lem:infimum_long_edges}
Let $\G$ be a metric graph. Assume that, for every $\ell > 0$,  $\G$
has an edge $e_\ell$ of length at least $\ell$.
Then there exists a sequence of functions
$(u_\ell)_{\ell \in \N} \subseteq \mathcal{N}_{\lambda}$,
equal to zero outside $e_\ell$, and such that
\begin{equation*}
\lim_{\ell \rightarrow \infty} J_{\lambda}(u_\ell) = s_{\lambda}.
\end{equation*}
In particular
\begin{equation}
\label{dis1}
\inf_{u\in  \NN_\lambda(\G)} J_{\lambda}(u) \le s_\lambda.
\end{equation}
\end{lemma}

\begin{proof}
Let $\phi_\lambda$ be the soliton in $\NN_\lambda(\R)$ and set $\delta_\ell = \phi_\lambda(\ell/2) = o(1)$ as $\ell\to \infty$. 
For every $\ell>0$, identify  the interval $[-\ell/2,\ell/2]$ with a subset of the edge $e_\ell$ and define $v_\ell \in H^1(\G)$ as
\[
v_\ell(x) = \begin{cases}  \left(\phi_\lambda(x) - \delta_\ell\right)^+ & \text{ if } x\in e_\ell, \\
0 & \text{ elsewhere on } \G. \end{cases}
\]
Since $\|v_\ell\|_{H^1(\G)} = \|v_\ell\|_{H^1(e_\ell)} = \|\phi_\lambda\|_{H^1(\R)} +o(1)$ as $\ell \to \infty$, and likewise for all the $L^q$ norms,
we see that $\pi_\lambda(v_\ell) \to 1$ as $\ell\to \infty$. Therefore, as $\pi_\lambda(v_\ell)v_\ell \in \mathcal{N}_\lambda(\G)$, 
\[
J_{\lambda}(\pi_\lambda(v_\ell)v_\ell)= J_{\lambda}(v_\ell) +o(1) = J_{\lambda}(\phi_\lambda) + o(1) = s_\lambda + o(1)
\]
as $\ell\to \infty$, and we conclude.
\end{proof}

In \cite{AST1} the authors introduced a topological condition on $\G$ under which  inequality \eqref{dis1} is reversed.
In our setting this condition, that we call assumption (H) as in \cite{AST1}, takes the following form.

\begin{definition}
\label{condH} We say that a metric graph  $\G\in {\bf G}$ satisfies assumption (H) if,
for every point $x_0\in\G$, there exist two injective curves
$\gamma_1,\gamma_2:[0,+\infty)\to\G$ parameterized by arclength, with disjoint
images except for an at most countable number of  points, and such that $\gamma_1(0)=\gamma_2(0)=x_0$.
\end{definition}

If a graph $\G\in {\bf G}$ satisfies assumption (H), it is easy to see (and proved in \cite{AST1}) that for every nonnegative $u \in H^1(\G)$
\[
\#u^{-1}(t) \ge 2\qquad \text{for almost every } t \in (0,\|u\|_\infty).
\]
Therefore, letting $\widehat u$ be the symmetric rearrangement of $u$, if $u \in \NN_\lambda(\G)$ we have $\pi_\lambda(\widehat u) \le 1$ by \eqref{equisym}. 
As $\pi_\lambda(\widehat u)\widehat u \in \NN_\lambda(\R)$, we conclude by \eqref{infR} and \eqref{formJ} that
\[
s_\lambda \le J_{\lambda}(\pi_\lambda(\widehat u)\widehat u) = \kappa\pi_\lambda(\widehat u)^p\|\widehat u\|_{L^p(\R)} ^p\le \kappa \| u\|_{L^p(\G)}^p = J_{\lambda}(u)
\]
and since this holds for every $u \in \NN_\lambda(\G)$, inequality \eqref{dis1} is in fact an equality.

As a consequence, with the same techniques as in \cite{AST1}, it is easy to prove the following result.

\begin{theorem}
\label{notatt}
If $\G\in {\bf G}$ satisfies assumption (H),  then
\[
\inf_{u\in  \NN_\lambda(\G)} J_{\lambda}(u) = s_\lambda
\]
but it is never achieved, unless $\G$ is isometric to $\R$ or to one of the exceptional graphs depicted in Example 2.4 of \cite{AST1}.
\end{theorem}

If more is known on the number of preimages of a function $u\in \NN_\lambda(\G)$, one can obtain sharper estimates. The following result will be used in Section \ref{sec:proof2}.

\begin{proposition}
\label{sK2} 
Let $\G\in {\bf G}$ and assume that $u\in \NN_\lambda(\G)$ is nonnegative and  satisfies
\[
\# u^{-1}(t) \ge K \qquad\text{for a.e.  $t\in (\inf_{\G} u, \sup_{\G} u)$}
\]
for some integer $K \ge 1$. Then
\[
J_{\lambda}(u) \ge K\frac{s_\lambda}2.
\]
\end{proposition}

\begin{proof} Since $\G\in {\bf G}$ we have $|\G| = +\infty$. By Proposition \ref{Ncontr} applied to $T=(\inf_{\G} u, \sup_{\G} u)=(0, \sup_{\G} u)$, the decreasing rearrangement $u^*$ of $u$ satisfies  $u^*\in H^1(\R^+)$ and
\[
\|u^*\|_{L^q(\R^+)} =  \|u\|_{L^q(\G)}, \qquad \text{for every $q\in [1,+\infty]$}, \qquad \|(u^*)'\|_{L^2(\R^+)} \le \frac1K \|u'\|_{L^2(\G)}.
\]
Setting $u_K(x) := u^*(Kx)$, a standard change of variable shows that
\begin{align*}
\pi_\lambda(u_K)^{p-2} &= \frac{\|u_K'\|_{L^2(\R^+)}^2 + \lambda \|u_K\|_{L^2(\R^+)}^2}{\|u_K\|_{L^p(\R^+)}^p} =
\frac{K \|(u^*)'\|_{L^2(\R^+)}^2 + \frac{\lambda}K \|u^*\|_{L^2(\R^+)}^2}{\frac1K \|u^*\|_{L^p(\R^+)}^p} \\ 
&\le \frac{\|u'\|_{L^2(\G)}^2 +  \lambda\|u\|_{L^2(\G)}^2}{\|u\|_{L^p(\G)}^p} =1
\end{align*}
and therefore, as $\pi_\lambda(u_K)u_K \in \NN_\lambda(\R^+)$,  by \eqref{infR+}
\[
\frac{s_\lambda}2 \le J_{\lambda}(\pi_\lambda(u_K)u_K ) = \kappa \pi_\lambda(u_K)^p \|u_K\|_{L^p(\R^+)} ^p\le  \frac{\kappa}{K} \|u\|_{L^p(\G)}^p = \frac1K J_{\lambda}(u).
\]
\end{proof}

\section{A general existence result}
\label{sec:doppiovinc}

In this section we prove a general existence result for positive solutions to \eqref{NLS} attaining their maximum in the interior of a prescribed sufficiently long edge.
Throughout this section, we work with metric graphs in class {\bf G}, most of the time requiring also that the graph satisfies assumption (H).
However, the arguments described in this section can be adapted with minor modifications to cover even broader classes of graphs.

Let $\G\in \bf{G}$ satisfy assumption (H) and let $e$ be one of its bounded edges. Set
\[
X_e:=\left\{u\in H^1(\G)\,\mid\, \|u\|_{L^\infty(\G)}=\|u\|_{L^\infty(e)}\right\}
\]
and consider the doubly--constrained minimization problem
\begin{equation}
\label{minXe}
c_\lambda(\G,e) := \inf_{u\in\mathcal{N}_\lambda(\G)\cap X_e}J_{\lambda}(u).
\end{equation}
\begin{remark}
\label{remVcl} The set $X_e$ is closed in the weak topology of $H^1(\G)$. Indeed,
if $u_n\in X_e$ and $u_n \rightharpoonup u$ in $H^1(\G)$, then  by semicontinuity
\[
\|u\|_{L^\infty(\G)} \leq \liminf_{n} \|u_n\|_{L^\infty(\G)} =\liminf_{n} \|u_n\|_{L^\infty(e)}= \|u\|_{L^\infty(e)}\,,
\]
the last equality being justified by the uniform convergence of $u_n$ to $u$ on $e$.
\end{remark}

The next two theorems state the main results of this section.
\begin{theorem}
\label{exlarge}
There exists $\overline{R}>0$ depending only on $\lambda$ and $p$
such that, if $\G\in {\bf G}$ satisfies assumption (H) and  has
a bounded edge $e$ of length $R \ge \overline{R}$, then $c_\lambda(\G,e)$ is attained. 
\end{theorem}

\begin{theorem}
	\label{usol}
	Let $\G\in {\bf G}$ satisfy assumption (H)
and have  a bounded edge $e$ of length $R$. Let $\ell_0\leq \displaystyle \inf_{e \in \mathbb E} |e|$.
Then there exists $\widetilde{R}$ 
	depending only on $\ell_0$, $\lambda$ and $p$
	such that if $R \ge \widetilde{R}$ and $u$ is a minimizer for $c_\lambda(\G,e)$, then $u\in\Sf_{\lambda}(\G)$
	and $u > 0$ or $u < 0$ on $\G$. Moreover,
	\begin{equation*}
		\| u \|_{L^{\infty}(e)} > \| u \|_{L^{\infty}(\G \setminus e)}.
	\end{equation*}
\end{theorem}

\begin{remark}
It is interesting to observe the difference of dependence of the threshold in the two previous results. In the first one, there is no dependence on the rest of the graph while in the second one, there is a slight dependence on the infimum of the length of the other edges.
\end{remark}

A straightforward consequence of the preceding theorems is the following multiplicity result.

\begin{theorem}
\label{multiple}
Under the assumptions of Theorem \ref{usol}, 
there exists $\widetilde{R} > 0$
depending only on $\ell_0$,  $\lambda$ and $p$
such that 
for every bounded  edge $e$ of length larger than $\widetilde{R}$, 
problem \eqref{NLS} has a positive solution attaining its absolute maximum on $e$ only.
Hence, 
if $\G$ has $n$ bounded edges of length greater than $\widetilde{R}$,
then problem \eqref{NLS} has at least $n$ distinct positive solutions.
\end{theorem}

For the proof of Theorems \ref{exlarge}--\ref{usol} we need the following lemmas. 

\begin{lemma}
\label{Lem 3.5}
Let $\G\in {\bf G}$ satisfy assumption (H). 
If $(u_n)_n$ is a bounded sequence in $H^1(\G)$
such that $\displaystyle\liminf_n \| u_n \|_{p} > 0$
and $\displaystyle\lim_n L(u_n) = \theta > 0$,  then 
\begin{equation*}
\liminf_{n \rightarrow \infty} \| u_n \|_p^p
	\ge  \frac{s_{\theta}}{\kappa},
\end{equation*}
where $s_\theta$ is defined in \eqref{solitonlevel}.
\end{lemma}

\begin{proof}
Let $\theta_n := L(u_n)$, so that $u_n \in \mathcal{N}_{\theta_n}(\G)$
and $\theta = \displaystyle\lim_n \theta_n$. Then, since $(u_n)_n$ is bounded in $L^2(\G)$, $\displaystyle\liminf_n \|u_n\|_{p} > 0$ and $\displaystyle\lim_n \theta_n=\theta$,  we have
\begin{align*}
\pi_{\theta}(u_n)^{p-2} &=  \frac{\|u_n'\|_{2}^{2} + \theta \|u_n\|_{2}^{2}}{\|u_n\|_{p}^{p}}
= \frac{\|u_n'\|_{2}^{2} + \theta_n \|u_n\|_{2}^{2}}{\|u_n\|_{p}^{p}}
+ (\theta - \theta_n)\frac{ \|u_n\|_{2}^{2}}{\|u_n\|_{p}^{p}} \\
&= 1 +  (\theta - \theta_n)\frac{\|u_n\|_{2}^{2}}{\|u_n\|_{p}^{p}},
\end{align*}
whence
\[
\lim_n \pi_{\theta}(u_n)=1.
\]
By Theorem \ref{notatt}, we conclude that
\[
s_{\theta} = \inf_{v \in \mathcal{N}_\theta(\G)} J_{\theta}(v) \le \liminf_n J_{\theta}(\pi_{\theta}(u_n) u_n)
= \liminf_n \kappa \pi_{\theta}(u_n)^p  \| u_n \|_{p}^{p}
= \liminf_n \kappa \| u_n \|_{p}^{p}
\]
and the result follows.
\end{proof}

The next result describes the behavior of minimizing sequences for problem \eqref{minXe}.

\begin{lemma}
\label{prova} 
Let  $\G\in {\bf G}$  and let $e\in \G$ be a bounded edge.
Let $(u_n)_n \subset \NN_\lambda(\G) \cap X_e$ be any minimizing sequence for $J_{\lambda}$ in $\NN_\lambda(\G) \cap X_e$.
Then $(u_n)_n$ admits a subsequence (not relabeled) such that 
\begin{itemize}
\item[1)] $u_n \rightharpoonup u$ in $H^1(\G)$, $u_n \to u$ in $L^q_{loc}(\G)$ for every $q \in [1,+\infty]$, 
$\inf \| u_n \|_p > 0$; 	
\item[2)] $\displaystyle\lim_{n}\|u_n\|_2^2 = \mu>0$;
\item[3)] $u \in X_e \setminus\{0\}$;
\item[4)]  $L(u) \le \lambda$;
\item[5)] if $L(u) = \lambda $, then $u$ is a minimizer for $J_{\lambda}$ on $\NN_\lambda(\G) \cap X_e$;
\item[6)] if $L(u) < \lambda$, then $m:=\|u\|_2^2 <\mu$ and
\begin{equation}
\label{limitL}
\lim_n L(u_n-u) = \lambda +\frac{m}{\mu- m}(\lambda- L(u))\,.
\end{equation}
\end{itemize}
\end{lemma}

\begin{proof} Since $(u_n)_n$ is a minimizing sequence, up to subsequences, {\em 1)} and {\em 2)} can be deduced from  Proposition \ref{boundedness}.

Notice that $u\not\equiv 0$ since if this were not the case, by $L^\infty_{loc}$ convergence,
and hence convergence in $L^{\infty}(e)$,
\[
\|u_n\|_{L^p(\G)}^p \le \|u_n\|_{L^\infty(\G)}^{p-2} \|u_n\|_{L^2(\G)}^2= \|u_n\|_{L^\infty(e)}^{p-2}  \|u_n\|_{L^2(\G)}^2\to 0,
\]
violating  {\em 1)}. Moreover, since $X_e$ is weakly closed (see Remark \ref{remVcl}), $u\in X_e$, and {\em 3)} is proved.
 
To prove {\em 4)} observe that by lower semicontinuity, since $\pi_\lambda(u)u \in \NN_\lambda(\G) \cap X_e$, the estimate
\begin{equation}
\label{pi1}
\kappa \| u \|_{p}^p \le \liminf_n \kappa \| u_n\|_{p}^p = \liminf_n J_{\lambda}(u_n) = c_\lambda(\G,e)
\le  J_{\lambda}(\pi_{\lambda}(u)u)
= \kappa \pi_{\lambda}(u)^p \| u \|_{p}^p
\end{equation}
yields $\pi_{\lambda}(u) \ge 1$, which is equivalent to $L(u) \le \lambda$.

Now if $L(u) = \lambda$, then $u \in \NN_\lambda(\G) \cap X_e$ and \eqref{pi1} shows that $u$ is a minimizer, which is {\em 5)}.

Finally, suppose that $L(u) < \lambda$ and hence $u_n - u \not\equiv 0$ for all $n$ large.
Obviously, by semicontinuity, $m \le \mu$.
To prove that the inequality is strict, first observe that by the Gagliardo--Nirenberg inequality \eqref{GN},
\begin{align}
\label{Lbounded}
L(u_n - u) &= \frac{\|u_n-u\|_{p}^p- \|u_n'-u'\|_{2}^2}{\|u_n-u\|_{2}^2} \le  \frac{\|u_n-u\|_{p}^p}{\|u_n-u\|_{2}^2} \le 
\frac{K\|u_n-u\|_{2}^{\frac{p}2+1} \|u_n'-u'\|_{2}^{\frac{p}2-1}}{\|u_n-u\|_{2}^2}
 \nonumber \\
& = K \|u_n-u\|_{2}^{\frac{p}2-1} \|u_n'-u'\|_{2}^{\frac{p}2-1} \le C,
\end{align}
for every $n$, since $(u_n)_n$ is bounded in $H^1(\G)$. Now, by the Brezis--Lieb Lemma \cite{BL}, as $n\to \infty$,
\begin{align*}
L(u_n - u) &= \frac{\|u_n-u\|_{p}^p- \|u_n'-u'\|_{2}^2}{\|u_n-u\|_{2}^2} = 
 \frac{\|u_n\|_{p}^p- \|u_n'\|_{2}^2- \|u\|_{p}^p+ \|u'\|_{2}^2+ o(1)}{\|u_n\|_{2}^2 - \|u\|_{2}^2+o(1) } \nonumber \\
&= \frac{\lambda \|u_n\|_{2}^2 -L(u) \|u\|_{2}^2 + o(1)}{\|u_n\|_{2}^2 - \|u\|_{2}^2+o(1)} = \lambda + 
\frac{(\lambda-L(u))\|u\|_2^2 +o(1)}{\|u_n\|_{2}^2 - \|u\|_{2}^2+o(1)}.
\end{align*}
Since $L(u) < \lambda$ and $u \not\equiv 0$, we see that $\mu = \lim_n\|u_n\|_2^2 > \|u\|_2^2 = m$,
since otherwise \eqref{Lbounded} is violated. Letting $n\to \infty$ in the preceding equality, we obtain \eqref{limitL} and the proof is complete.
\end{proof}

\begin{remark}
\label{level} 
As a consequence of Lemma \ref{lem:infimum_long_edges}, for every $\eps>0$ there exists $R_\eps>0$ such that, if $\G$ is any metric graph containing
a bounded edge $e$ of length greater than $R_\eps$, then
\begin{equation*}
c_{\lambda}(\G,e) \le  s_\lambda +\eps.
\end{equation*}
\end{remark}

We are now ready to prove the main results of this section.

\begin{proof}[Proof of Theorem \ref{exlarge}]
Fix $\eps>0$ such that
\begin{equation}
\label{choixeps}
s_{\lambda +\eps} = s_1(\lambda + \eps)^\alpha \le  2s_1 \lambda^\alpha = 2s_\lambda   , \qquad \lambda-C \eps>0 \quad\mbox{ and } \quad(\lambda + \eps)^\alpha< \lambda^\alpha + (\lambda -C \eps)^\alpha,
\end{equation}
with 
\begin{equation}
\label{DefC}
C=16 s_1^2/\kappa^2.
\end{equation}
Observe that $C$ depends only on $p$ and, hence, that $\eps$ depends only on $p$ and $\lambda$ but not on $\G$.

Let $\overline R$ be so large that for every $R \ge \overline R$,
\[
c_\lambda(\G,e) < s_{\lambda +\eps},
\]
which is possible by Remark \ref{level}. Again, we observe that $\overline R$ depends only on $p$ and $\lambda$ but not on $\G$.

Let $(u_n)_n \subset  \NN_\lambda(\G)\cap X_e$ be a minimizing sequence for $J_{\lambda}$ such that
\[
J_{\lambda}(u_n) \le s_{\lambda +\eps}
\]
for every $n$. Applying Lemma \ref{prova}, $(u_n)_n$ has (up to subsequences) a weak limit $u \in X_e \setminus\{0\}$ such that $L(u) \le \lambda$ and,
in case equality holds, $u$ is the required minimizer. 

We now show that $L(u) < \lambda$ cannot happen, and this will end the proof. If $L(u) < \lambda$, by \eqref{limitL},
\[
\lim_n L(u_n-u)
= \lambda + \frac{m}{\mu- m}(\lambda- L(u))
> \lambda > 0, 
\]
with $0 < m = \|u\|_2^2 < \lim_n \|u_n\|_2^2 = \mu$.
We note that
\begin{equation}
	\label{u_n_not_cvg_Lp}
	\liminf_n \| u_n - u \|_{p} > 0,
\end{equation}
as otherwise,  up to a subsequence,  $(u_n)_n$ converges to $u$ in $L^p(\G)$,  and by lower semicontinuity we would have 
\begin{equation*}
	\lambda
	> L(u)
	= \frac{\| u \|_p^p - \| u' \|_2^2}{\| u \|_2^2}
	\ge \liminf_{n} \frac{\| u_n \|_p^p - \| u_n' \|_2^2}{\| u_n \|_2^2}
	= \lambda,
\end{equation*}
a contradiction.

As $\G\in {\bf G}$ satisfies  assumption (H),
\eqref{u_n_not_cvg_Lp}, \eqref{limitL} and Lemma~\ref{Lem 3.5} give
\begin{equation}
\label{levone}
\liminf_n \|u_n -u\|_p^p \ge \frac{s_1}{\kappa}\left(\lambda +\frac{m}{\mu-m}(\lambda- L(u))\right)^\alpha.
\end{equation}
Notice that 
\begin{equation}
\label{boundmu}
\lim_n  \left(\| u_n' \|_{2}^2 +\lambda  \| u_n \|_{2}^2 \right)= \frac1{\kappa} \lim_n J_{\lambda}(u_n) \le  \frac{s_{\lambda+\eps}}{\kappa} \le \frac{2s_1}{\kappa}\lambda^{\alpha}
\end{equation}
and hence 
\begin{equation}
\label{boundmubis}
\mu = \lim_n  \| u_n \|_{2}^2 \le \lim_n \f1\lambda \left(\| u_n' \|_{2}^2 +\lambda  \| u_n \|_{2}^2 \right) \le\frac{2s_1}{\kappa} \lambda^{\alpha-1}. 
\end{equation}
Furthermore, as
\begin{equation}
\label{bounduinfty1}
\lambda \| u_n \|_{2}^2 \le \| u_n' \|_{2}^2 +\lambda  \| u_n \|_{2}^2 =  \| u_n \|_{p}^ p \le  \| u_n \|_{\infty}^{p-2}  \| u_n \|_{2}^2,
\end{equation}
we see that 
\begin{equation}
\label{bounduinfty2}
 \| u \|_{L^\infty(\G)} =  \| u \|_{L^\infty(e)} = \lim_n  \| u_n \|_{L^\infty(e)}  =\lim_n  \| u_n \|_{L^\infty(\G)}\ge \lambda^{\frac1{p-2}}
\end{equation}
and therefore, by the Gagliardo-Nirenberg inequality \eqref{GNinfty} and \eqref{boundmu}, recalling that $m=\|u\|_2^2$, we have
\begin{equation}
\label{boundm}
\lambda^{\frac4{p-2}} \le \| u \|_{\infty}^4 \le 4 m \| u' \|_{2} ^2 \le  4 m \liminf_n   \| u_n'\|_{2}^2 \le  \frac{4m}\kappa    s_{\lambda +\eps} \le \frac{8 s_1m}\kappa   \lambda^\alpha.
\end{equation}
Thus, recalling from  \eqref{solitonlevel} the value of $\alpha$, we see from \eqref{boundmubis} and \eqref{boundm} that
\begin{equation}
\label{boundrap}
\frac{\mu}m \le \frac{16s_1^2}{\kappa^2}\lambda^{2\alpha -1-\frac{4}{p-2}}  = 
C,
\end{equation}
with $C$ given by \eqref{DefC}. 

In conclusion, by the Brezis--Lieb Lemma and \eqref{levone},
\begin{align}
\label{final1}
s_1(\lambda +\eps)^\alpha
&\ge c_\lambda(\G,e)
= \lim_n J_{\lambda}(u_n)
= \lim_n \kappa \| u_n \|_p^p
= \lim_n \kappa \left( \| u_n - u \|_p^p + \| u \|_p^p \right)\nonumber \\
& \ge s_1 \left( \lambda +\frac{m}{\mu-m}(\lambda - L(u)) \right)^\alpha + \kappa  \| u \|_p^p.
\end{align}
Neglecting the last term, we obtain
\[
\eps \ge \frac{m}{\mu-m}(\lambda- L(u)),
\]
or, rearranging terms and using \eqref{boundrap}, 
\[
L(u) \ge \lambda - \frac{\mu- m}{m}\eps > \lambda - C\eps>0,
\]
Using this, we see from Lemma  \ref{Lem 3.5} that
$$
\kappa \|u\|_p^p \geq s_1 L(u)^{\alpha} \geq  s_1 ( \lambda - C\eps)^{\alpha}.
$$
Hence, by \eqref{final1},
\[
(\lambda +\eps)^\alpha \ge \lambda^\alpha + (\lambda- C\eps)^\alpha 
\]
which contradicts \eqref{choixeps}.
\end{proof}

\begin{remark}
\label{levels_close_to_s_lambda}
Before proving Theorem \ref{usol} we recall that, letting 
\[
\mathcal{D} := \Bigl\{u \in H^1(\R) \bigm\vert u \ge 0, u \text{ is even},
u \text{ is nonincreasing on } [0, +\infty) \Bigr\},
\]
for every $\eps >0$ there exists $\delta >0$ such that, for every $u \in \NL(\R) \cap \mathcal{D}$ with $J_\lambda(u) \le s_\lambda + \delta$, there results
\[
\| u - \phi_{\lambda} \|_{H^1} \le \eps.
\]
This can be readily seen by standard compactness arguments and the uniqueness of the minimizer of $J_\lambda$ in $\NL(\R) \cap \mathcal{D}$ (see e.g \cite{cazenave}).

\end{remark}

\begin{proof}[Proof of Theorem \ref{usol}] 
		By Remark \ref{levels_close_to_s_lambda}, there exists $\delta >0$ such that for every $u \in \NN_{\lambda}(\R) \cap \mathcal{D}$ with $J_{\lambda}(u) \le s_\lambda + \delta$ there results 
\begin{equation}
\label{Q0}
\| u - \phi_{\lambda} \|_{H^1} \le \frac13\| \phi_{\lambda}' \|_{L^2(-\ell_0/2, \ell_0/2)}.
\end{equation}
Let $0<\eps\leq \delta$ satisfy
\begin{equation}
\label{Q1}
\left[\frac45 \left(1-\left(\frac{s_{\lambda}}{s_{\lambda}+\eps}\right)^{\frac{p-2}{p}}\right) \frac{s_{\lambda}+\eps}{\kappa}  \right]^{1/2}
\leq \frac13\| \phi_{\lambda}' \|_{L^2(-\ell_0/2, \ell_0/2)}
\end{equation}
and, accordingly to  Remark \ref{level}, let $\widetilde R >0$ be 
such that
\begin{equation}
\label{Q2}
c_{\lambda}(\G,e) \le  s_\lambda +\eps
\end{equation}
for every edge $e$ with length $R \ge \widetilde R$.
Observe that $\widetilde R$ depends only on $\lambda$, $p$ and $\ell_0$.

Let $u_R\in\NN_\lambda(\G)\cap X_e$
	be a function satisfying $J_{\lambda}(u_R) = c_\lambda(\G,e)$.
	To show that $u_R\in\mathcal{S}_\lambda(\G)$, it is enough to prove that
	\begin{equation}
		\label{interno}
		\|u_R\|_{L^\infty(e)}>\|u_R\|_{L^\infty(\G\setminus e)}.
	\end{equation}
	Indeed, \eqref{interno} implies that $u_R$ belongs
	to the relative interior of $\NN_\lambda(\G)\cap X_e$,
	and therefore it is not only a global minimizer of $J_{\lambda}$ in the double constraint space,
	but also a local minimizer in $\NN_\lambda(\G)$, and, as such, solves \eqref{NLS}.	Since $|u_R|\in\NN_\lambda(\G)\cap X_e$ and $J_{\lambda}(u_R)=J_{\lambda}(|u_R|)$,
	we will assume that $u_R \ge 0$ on $\G$.
	 
	We proceed by contradiction and assume that
	\[
	\|u_R\|_{L^\infty(e)}=\|u_R\|_{L^\infty(\G\setminus e)}.
	\]
Let $M_R := \|u_R\|_{L^\infty(e)}$, denote by $\BB$
	the set of all bounded edges of $\G$ and set
	\begin{equation}
		\label{defdelta}
		\delta_R:=\max_{h\in \BB} \,\min_{x\in h} u_R(x)\qquad\quad (0\leq\delta_R\leq M_R).
	\end{equation}
	The definition of $\delta_R$ as a maximum is correct even
	if $\BB$ contains infinitely many edges.
	Indeed, as $\|u_R\|_{2}$ is finite
	and the length of the edges is bounded from below by $\ell_0$,
	there is only a finite number of edges $h$ where $\min_h u_R \ge t$, for every $t >0$.
	
	Since $\G$ satisfies assumption (H),  by Step 1 of the proof of Lemma 4.2 of \cite{ASTbound},
	\begin{equation}
		\label{3ormore}
		\#u_R^{-1}(t) \ge 3 \quad\text{ for almost every $t \in [\delta_R, M_R]$}.
	\end{equation}
	Notice that the set
	\[
	A_R := \big\{x \in \G \mid u_R(x) \in [\delta_R, M_R] \big\}
	\]
	contains at least one bounded edge of $\BB$
	(the one where the minimum of $u_R$ is exactly $\delta_R$) and therefore
	we have 
	\[
	\big | A_R \big | \ge \ell_0.
	\]
Let now $\widehat u_R$ be the symmetric rearrangement of $u_R$.
By Proposition \ref{Ncontr} applied to $T_R= [\delta_R, M_R]$  and \eqref{3ormore}, defining 
$\ell_R=\big | A_R \big |/2$,
we have
	\begin{equation}
		\label{strongPS}
		\Vert \widehat u_R'\Vert_{L^2(-\ell_R, \ell_R)}^2 = 4\| (u^*_R)' \|_{L^2(0,2\ell_R)}^2  \leq\frac {4}{9}
		\Vert u_R'\Vert_{L^2(A_R)}^2.
	\end{equation}
	Next, denoting\footnote{If $\delta_R = 0$, the set $B_R$ is empty,
		$A_R = \G$, and the proof is simpler,
		working with $A_R$ only. }
	by $B_R = \big\{x \in \G \mid u_R(x) \in [0, \delta_R) \big\}$, since by assumption (H)
	\[
	\#u_R^{-1}(t) \ge 2 \quad\text{ for almost every $t \in (0, \delta_R)$},
	\]
	we obtain 
	\[
	\|\widehat u_R' \|_{L^2(\R\setminus (-\ell_R,\ell_R))}
	 \le \| u_R'\|_{L^2(B_R)}.
	\]
	From these relations it follows
	\[
	\|\widehat u_R'\|_{L^2(\R)}^2  \le \frac49 \| u_R'\|_{L^2(A_R)}^2 +\| u_R'\|_{L^2(B_R)}^2= \| u_R'\|_{L^2(\G)}^2 -\frac59 \| u_R'\|_{L^2(A_R)}^2
	\]
	so that
	\begin{equation}
		\label{major_pi_w_R}
		\pi_\lambda(\widehat u_R)^{p-2}
		\le
		\frac{\|u_R'\|_{L^2(\G)}^2
			+ \lambda \|u_R\|_{L^2(\G)}^2
			-\frac59  \| u_R'\|_{L^2(A_R)}^2}{\|u_R\|_{L^p(\G)}^p }
		= 1 - \frac59\frac{\| u_R'\|_{L^2(A_R)}^2}{\|u_R\|_{L^p(\G)}^p }.
	\end{equation}
		Since $\pi_\lambda(\widehat u_R)\widehat u_R \in \NN_\lambda(\R)$, by \eqref{infR} we obtain
	\begin{equation}
		\label{ineq_levels_w_R}
		s_\lambda
		\le J_{\lambda}(\pi_\lambda( \widehat u_R) \widehat u_R)
		= \kappa \pi_\lambda( \widehat u_R )^p \| \widehat u_R \|_p^p
		= \kappa \pi_\lambda( \widehat u_R )^p \| u_R \|_p^p
		= \pi_\lambda( \widehat u_R )^p J_{\lambda}(u_R).
	\end{equation}
	Using \eqref{ineq_levels_w_R} and recalling that $J_{\lambda}(u_R) = c_\lambda(\G,e)$,
	\eqref{Q2} and \eqref{major_pi_w_R} imply that
	\begin{equation}
		\label{level_close_to_soliton}
		s_{\lambda} \le J_{\lambda}(\pi_\lambda(\widehat u_R) \widehat u_R) \le s_{\lambda} + \eps.
	\end{equation}
		Since, by definition $\pi_\lambda(\widehat u_R) \widehat u_R$ belongs to $\mathcal{D}$, with $\mathcal{D}$ defined in Remark \ref{levels_close_to_s_lambda},
	\eqref{level_close_to_soliton},
	 \eqref{Q0}
and the fact that $\eps\leq \delta$ imply that
	\begin{equation*}
		\| \phi_{\lambda} - \pi_\lambda(\widehat u_R) \widehat u_R \|_{H^1(\R)} 
		\le \frac13\| \phi_{\lambda}' \|_{L^2(-\ell_0/2, \ell_0/2)},
	\end{equation*}
	which itself implies
	\begin{equation}
		\label{contrad}
		\| \phi_{\lambda}' \|_{L^2(-\ell_0/2, \ell_0/2)}
		\le \frac{\| \phi_{\lambda}' \|_{L^2(-\ell_0/2, \ell_0/2)}}{3} 
		+ \| \pi_\lambda(\widehat u_R) \widehat u_R' \|_{L^2(-\ell_0/2, \ell_0/2)}.
	\end{equation}
In order to obtain a contradiction, 	we now prove  that 
		\begin{equation}
		\label{small_deriv_norm_w_R}
		\| \pi_\lambda(\widehat u_R) \widehat u_R'\|_{L^2(-\ell_0/2, \ell_0/2)} \le  
		\frac{\| \phi_{\lambda}' \|_{L^2(-\ell_0/2, \ell_0/2)}}{3}.
	\end{equation}
	By \eqref{ineq_levels_w_R}, we have
	\[
	\frac{s_\lambda}{J_{\lambda}(u_R)} \le \pi_\lambda( \widehat u_R )^p.
	\]
	Using the previous inequality, \eqref{Q2}
	and \eqref{major_pi_w_R}, we deduce that
	\begin{equation}
		\label{control_pi_lambda}
		\Big(\frac{s_{\lambda}}{s_{\lambda}+\eps} \Big)^{1/p}\le \pi_\lambda( \widehat u_R ) \le 1.
	\end{equation}
	Using \eqref{major_pi_w_R}, \eqref{control_pi_lambda} and  $\|u_R\|_{L^p(\G)}^p=\frac{c_\lambda(\G,e)}{\kappa}
	\leq \frac{s_{\lambda}+\eps}{\kappa}$, we obtain 
	\begin{equation}
	\label{small_deriv_norm_u_R}
		\| u_R'\|_{L^2(A_R)}^2 
		\le 	\left(1-\left(\frac{s_{\lambda}}{s_{\lambda}+\eps}\right)^{\frac{p-2}{p}}\right) \frac95 \|u_R\|_{L^p(\G)}^p
		\leq \frac95\left(1-\left(\frac{s_{\lambda}}{s_{\lambda}+\eps}\right)^{\frac{p-2}{p}}\right) \frac{s_{\lambda}+\eps}{\kappa}.
	\end{equation}
	By \eqref{Q1}, \eqref{strongPS}, \eqref{control_pi_lambda}, \eqref{small_deriv_norm_u_R}
	and the fact that $2\ell_R \ge \ell_0$ for every $n$, we obtain 
\[
\| \pi_\lambda(\widehat u_R) \widehat u_R'\|_{L^2(-\ell_0/2, \ell_0/2)} \le \left[\frac45
\left(1-\left(\frac{s_{\lambda}}{s_{\lambda}+\eps}\right)^{\frac{p-2}{p}}\right) \frac{s_{\lambda}+\eps}{\kappa}\right]^{1/2}
\leq \frac{\| \phi_{\lambda}' \|_{L^2(-\ell_0/2, \ell_0/2)}}{3}
\]
which concludes the proof of \eqref{interno}.
\medbreak

The proof that $u_R>0$ in $\G$ follows by a strong maximum principle as in \cite[Proposition 3.3]{AST1} knowing that $u_R\geq 0$ on $\G$.
\end{proof}

\section{Proof of Theorem \ref{main}}
\label{sec:proof2}

This section is devoted to the proof of the main result of the paper.
In fact, alternatives A1 and B1 are straightforward, whereas A2 will follow as a direct application of the results proved in Section \ref{sec:doppiovinc}.
Conversely, case B2 is the most involved and will occupy the largest part of this section. 

For the sake of completeness, each case A1--B2 is presented here independently from the others.

\subsection{Case A1: $c_\lambda=\sigma_{\lambda}$, attained}
\label{subsec:A1}
As already pointed out in the Introduction, this is the easiest case and there is essentially nothing to say.
It is the case where $c_\lambda(\G)$ is attained by an action ground state, which is of course also a least action solution.
Straightforward examples for this are compact graphs, where ground states exist for every value of $\lambda$ and $p$
(see for example \cite{D}), but many graphs realizing alternative A1 can be identified also in the noncompact setting, as e.g. those presented in \cite[Section 3]{AST2}.

\subsection{Case A2: $c_\lambda=\sigma_{\lambda}$, not attained}
\label{subsec:A2}

The proof of this alternative is one of the main results of the paper and it relies heavily on Theorems \ref{exlarge}--\ref{usol} above. To exhibit a graph where A2 occurs we will use the following construction.

On a real line we insert, for each integer $k \ge 1$, a node $\vv_k$ at the point of coordinate $k$. At each $\vv_k$ we attach a self-loop $\LL_k$ of length $k$, by identifying $\vv_k$ with the only vertex of $\LL_k$. We obtain in this way the graph depicted in Figure \ref{bigcircles}. 

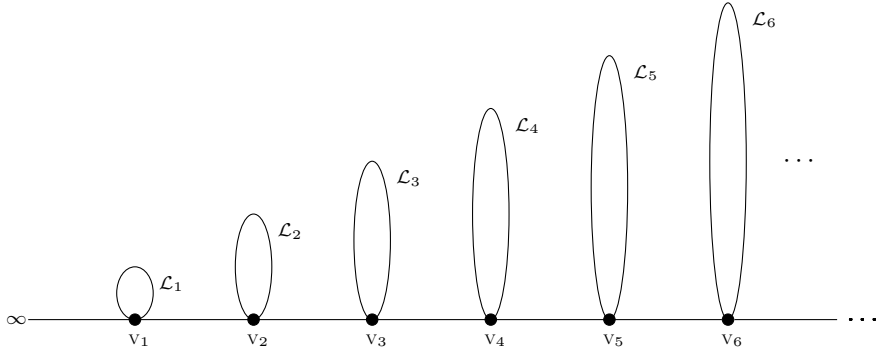
\begin{figure}[ht]
	\centering
	\begin{tikzpicture}[xscale=1.2,yscale=0.7]
	\node at (-1, 0) [infinito]  (0) {$\scriptstyle\infty$};
	\node at (8, 0) [infinito]  (100) {};
	\foreach \i in {1, ..., 6}
	{
		\node at (1.3*\i - 1, 0) [nodo] (1) {};
		\node at (1.3*\i - .95, -.1) [below] {$\scriptstyle \vv_\i$};
		\draw(1.3*\i - 1, 0.5*\i) ellipse (0.2 and 0.5*\i);
		\node at (1.3*\i-.6, \i) [below] {$\scriptstyle\LL_{\i}$};
	}
	\foreach \i in {1, ..., 7}
	\node at (8.3, 0) [infinito] {$\cdots$};
	\node at (7.6, 3) {$\cdots$};
	
	\draw [-] (0) -- (1);
	\draw [-] (100) -- (1);
	\end{tikzpicture}
	\caption{The graph $\G$ of Theorem \ref{Gc}}
	\label{bigcircles}
\end{figure}

\begin{theorem}
	\label{Gc}
	Let $\G$ be the graph in Figure \ref{bigcircles}. For every $\lambda>0$, 
	\begin{equation*}
	c_{\lambda}(\G) = \sigma_{\lambda}(\G) = s_{\lambda}
	\end{equation*}
	and neither $c_{\lambda}(\G)$ nor $\sigma_{\lambda}(\G)$ is attained.
\end{theorem}
\begin{proof}
	On the one hand, $\G\in {\bf G}$ satisfies assumption (H) by construction, so that $c_\lambda(\G)=s_\lambda$ and $c_\lambda(\G)$ is not attained by Theorem \ref{notatt}. On the other hand, Theorems \ref{exlarge}--\ref{usol} ensure that, for sufficiently large $k$, there exists $u_k\in\mathcal{S}_\lambda(\G)$ such that $J_{\lambda}(u_k)=c_\lambda(\G,\mathcal{L}_k)$. Hence, by Remark \ref{level},
	$$
	\sigma_{\lambda}(\G)\leq\liminf_{k\to\infty}J_{\lambda}(u_k)\leq s_\lambda,
	$$ 
	in turn implying $\sigma_{\lambda}(\G)=c_\lambda(\G)$ and concluding the proof.
\end{proof}

\subsection{Case B1: $c_\lambda(\G)<\sigma_\lambda(\G)$, $\sigma_{\lambda}(\G)$ attained}
\label{subsec:B1}
As anticipated in the Introduction, in view of the existing literature it is easy to produce graphs realizing alternative B1.
Indeed, it is for instance enough to let $\G$ be a star graph, i.e. a graph made of a finite number $N\geq3$ of half-lines, glued together at their common origin.
Star graphs have been widely investigated, as they provide the simplest example of noncompact graphs with half-lines.
On the one hand, since star graphs satisfy assumption (H), it is immediate to see that $c_\lambda(\G)=s_\lambda$ and it is not attained by Theorem \ref{notatt}.
On the other hand, one can exploit the simple structure of these graphs to characterize explicitly the set of all solutions $\mathcal{S}_\lambda(\G)$ (see e.g. \cite{ACFN1})
\begin{itemize}
			\item if $N$ is odd, \eqref{NLS} only admits two nonzero solutions $\pm u$,
			 the positive one given by a copy of the restriction of $\phi_\lambda$ to $\R^+$ on each half-line of the graph.
			
			\item if $N$ is even, the set of non-zero solutions
			of \eqref{NLS} is given by
			\begin{equation*}
				\{ \pm u_{I, a} \mid I \subset \{1, \cdots, N\}, \#I = N/2, a \in \R^+ \},
			\end{equation*}
			where 
			\begin{equation*}
				u_{I, a}(x)  = \begin{cases}
					\phi_{\lambda}(x+a)	&\text{if } x \in \HH_i, \text{ for some } i \in I,\\
					\phi_{\lambda}(x-a)	&\text{otherwise,}
				\end{cases}
			\end{equation*}
			with $\HH_i$ being the $i$--th half-line of the graph.
		\end{itemize}
Hence, we easily deduce that  $\sigma_{\lambda}(\G)=\f N2s_\lambda>s_\lambda$ and it is attained for instance by a function $u\in\mathcal{S}_\lambda(\G)$
whose restriction to each half-line of the graph coincides with the restriction of the soliton $\phi_\lambda$ to $\R^+$. 

\subsection{Case B2: $c_\lambda(\G)<\sigma_\lambda(\G)$, neither attained}
\label{subsec:B2}

The discussion of alternative B2 requires a deeper analysis with respect to the other cases. To prove that this alternative actually occurs,
we will consider the one--parameter family of noncompact graphs constructed as follows.

\begin{figure}[ht]
\centering
\subfloat[][]{
\begin{tikzpicture}[xscale= 0.5,yscale=0.5]
\draw (3,-4) -- (2,-6.5);
\draw (3,-4) -- (4,-6.5);
\node at (3,-4) [nodo] (00) {};
\node at (1.9,-6.7) [infinito] {$\scriptstyle\infty$};
\node at (4.1,-6.7) [infinito] {$\scriptstyle\infty$};
\node at (4.2,-5) [infinito] {$\RR$};
\node at (1.7,-4) [infinito] {$\phantom\RR$};

\end{tikzpicture}}\hskip 3 cm
\subfloat[][]{
\begin{tikzpicture}[xscale= 0.5,yscale=0.5]

\draw (0,-2) ellipse (1 and 2);
\node at (2,-1) [infinito] {$\LL$};
\node at (-1.7,-1) [infinito] {$\phantom\LL$};
\node at (0,-4) [nodo] (S) {};

\end{tikzpicture}}\hskip 3 cm
\subfloat[][]{
\begin{tikzpicture}[xscale= 0.5,yscale=0.5]
\draw (2,-1) circle (1); 
\node at (2,-2) [nodo] (S) {};
\node at (2,0) [nodo] (N) {};
\draw (2,-1) ellipse (.8 and 1);
\draw (2,-1) ellipse (.6 and 1);
\draw (2,-1) ellipse (.4 and 1);
\node at (3.5,0) [infinito] {$\BB$};
\node at (.8,0) [infinito] {$\phantom\BB$};
\end{tikzpicture}}

\caption{The building blocks of the graph $\G_N$. Two half-lines emanating from a vertex  (a); a self-loop of length $N$ (b); $N$ edges of length $1$ connecting two vetices (c).}
\label{blocks}
\end{figure}
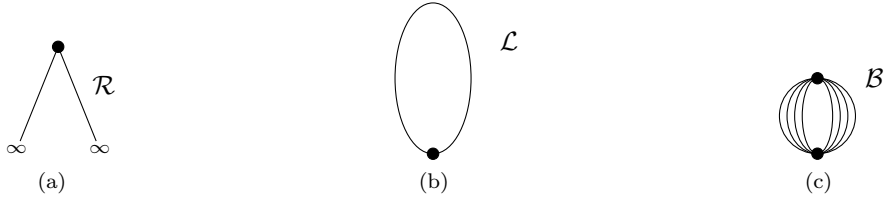

Let $N\in\N$ be fixed. Keeping in mind Figure \ref{blocks}, we consider a straight line on which we insert, for every $k\in \Z$, a vertex $\vv_k$ at the point of coordinate $k$. At each $\vv_k$ we attach a copy of the graph $\RR$, called $\RR_k$, by identifying $\vv_k$ with the vertex of $\RR_k$. Next, denoting by $\LL$ the self-loop of length $N$ in Figure \ref{blocks}(b),  we attach at each $\vv_k$ except $\vv_0$ a copy of the graph $\LL$, called $\LL_k$, by identifying $\vv_k$ with the only vertex of $\LL_k$. Finally, we attach at $\vv_0$ the graph $\BB$ with $N$ edges of length 1 in Figure \ref{blocks}(c), by identifying one of its two vertices with $\vv_0$. We call $\G_N$ the resulting graph, shown in Figure \ref{grafo1}.

\begin{figure}[ht]

\centering
\begin{tikzpicture}[xscale= 0.5,yscale=0.5]

\draw (1,-2) -- (23,-2);
\draw [dotted,thick] (-1,-2) -- (1,-2);
\draw [dotted,thick] (23,-2) -- (25,-2);

\foreach \x in {1,...,3}
\draw (\x*3,0) ellipse (1 and 2);
\foreach \x in {5,...,7}
\draw (\x*3,0) ellipse (1 and 2);

\foreach \x in {1,...,7}
\node at (\x*3,-2) [nodo] {};

\foreach \x in {1,...,7}
\draw (\x*3,-2) -- (\x*3-1,-6.5);
\foreach \x in {1,...,7}
\draw (\x*3,-2) -- (\x*3+1,-6.5);

\foreach \x in {1,...,7}
\node at (\x*3-1.1,-6.8) [infinito] {$\scriptstyle\infty$};
\foreach \x in {1,...,7}
\node at (\x*3+1.1,-6.8) [infinito] {$\scriptstyle\infty$};

\draw (12,-1) circle (1); 

\node at (12,0) [nodo] (N) {};
\draw (12,-1) ellipse (.8 and 1);
\draw (12,-1) ellipse (.6 and 1);
\draw (12,-1) ellipse (.4 and 1);

\node at (12.7,-2.4) [infinito] {$\scriptstyle\vv_0$};
\node at (15.7,-2.4) [infinito] {$\scriptstyle\vv_1$};
\node at (18.7,-2.4) [infinito] {$\scriptstyle\vv_2$};
\node at (21.7,-2.4) [infinito] {$\scriptstyle\vv_3$};

\node at (9.9,-2.5) [infinito] {$\scriptstyle\vv_{-1}$};
\node at (6.9,-2.5) [infinito] {$\scriptstyle\vv_{-2}$};
\node at (3.9,-2.5) [infinito] {$\scriptstyle\vv_{-3}$};

\node at (16,2) [infinito] {$\scriptstyle\LL_1$};
\node at (19,2) [infinito] {$\scriptstyle\LL_2$};
\node at (22,2) [infinito] {$\scriptstyle\LL_3$};

\node at (10.2,2) [infinito] {$\scriptstyle\LL_{-1}$};
\node at (7.2,2) [infinito] {$\scriptstyle\LL_{-2}$};
\node at (4.2,2) [infinito] {$\scriptstyle\LL_{-3}$};

\node at (13,0) [infinito] {$\scriptstyle\BB$};

\node at (4.3,-4.5) [infinito] {$\scriptstyle\RR_{-3}$};
\node at (7.3,-4.5) [infinito] {$\scriptstyle\RR_{-2}$};
\node at (10.3,-4.5) [infinito] {$\scriptstyle\RR_{-1}$};
\node at (13.1,-4.5) [infinito] {$\scriptstyle\RR_0$};
\node at (16.1,-4.5) [infinito] {$\scriptstyle\RR_1$};
\node at (19.1,-4.5) [infinito] {$\scriptstyle\RR_2$};
\node at (22.1,-4.5) [infinito] {$\scriptstyle\RR_3$};

\end{tikzpicture}
\caption{The graph $\G_N$.}
\label{grafo1}
\end{figure}

A second graph we will use below, that plays the role of ``limit graph'' with respect to $\G_N$, is depicted in Figure \ref{grafo2}. It is exactly equal to $\G_N$ except that the subgraph $\BB$ is replaced by a loop $\LL_0$ identical to all other loops.
Note that this graph is $\Z$--periodic. We call it $\widetilde\G_N$, and we label all its vertices, edges and subgraphs with the same letters as for those of $\G_N$, superposed by a tilde.

\begin{figure}[h]

\centering

\begin{tikzpicture}[xscale= 0.5,yscale=0.5]

\draw (1,-2) -- (23,-2);
\draw [dotted,thick] (-1,-2) -- (1,-2);
\draw [dotted,thick] (23,-2) -- (25,-2);

\foreach \x in {1,...,7}
\draw (\x*3,0) ellipse (1 and 2);
\foreach \x in {1,...,7}
\node at (\x*3,-2) [nodo] {};

\foreach \x in {1,...,7}
\draw (\x*3,-2) -- (\x*3-1,-6.5);
\foreach \x in {1,...,7}
\draw (\x*3,-2) -- (\x*3+1,-6.5);

\foreach \x in {1,...,7}
\node at (\x*3-1.1,-6.8) [infinito] {$\scriptstyle\infty$};
\foreach \x in {1,...,7}
\node at (\x*3+1.1,-6.8) [infinito] {$\scriptstyle\infty$};

\node at (12.5,-2.5) [infinito] {$\scriptstyle\widetilde\vv_0$};
\node at (15.5,-2.5) [infinito] {$\scriptstyle\widetilde\vv_1$};
\node at (18.5,-2.5) [infinito] {$\scriptstyle\widetilde\vv_2$};
\node at (21.5,-2.5) [infinito] {$\scriptstyle\widetilde\vv_3$};

\node at (9.7,-2.5) [infinito] {$\scriptstyle\widetilde\vv_{-1}$};
\node at (6.7,-2.5) [infinito] {$\scriptstyle\widetilde\vv_{-2}$};
\node at (3.7,-2.5) [infinito] {$\scriptstyle\widetilde\vv_{-3}$};

\node at (16,2) [infinito] {$\scriptstyle\widetilde\LL_1$};
\node at (19,2) [infinito] {$\scriptstyle\widetilde\LL_2$};
\node at (22,2) [infinito] {$\scriptstyle\widetilde\LL_3$};

\node at (10.2,2) [infinito] {$\scriptstyle\widetilde\LL_{-1}$};
\node at (7.2,2) [infinito] {$\scriptstyle\widetilde\LL_{-2}$};
\node at (4.2,2) [infinito] {$\scriptstyle\widetilde\LL_{-3}$};
\node at (13,2) [infinito] {$\scriptstyle\widetilde\LL_0$};

\node at (4.3,-4.5) [infinito] {$\scriptstyle\widetilde\RR_{-3}$};
\node at (7.3,-4.5) [infinito] {$\scriptstyle\widetilde\RR_{-2}$};
\node at (10.3,-4.5) [infinito] {$\scriptstyle\widetilde\RR_{-1}$};
\node at (13.1,-4.5) [infinito] {$\scriptstyle\widetilde\RR_0$};
\node at (16.1,-4.5) [infinito] {$\scriptstyle\widetilde\RR_1$};
\node at (19.1,-4.5) [infinito] {$\scriptstyle\widetilde\RR_2$};
\node at (22.1,-4.5) [infinito] {$\scriptstyle\widetilde\RR_3$};

\end{tikzpicture}

\caption{The graph $\widetilde\G_N$.}
\label{grafo2}	
\end{figure}

Note that the graph $\G_N$ is made of edges of length $1$ (the horizontal edges and the $N$ edges of $\BB$), of self-loops of length $N$  and of half-lines. Notice also, as it will be important in the sequel, that the total length of $\BB$ equals the length of the loop $\LL$.

Exploiting the dependence of $\G_N$ on the parameter $N$ we have the following result, which proves the actual occurrence of case B2.

\begin{theorem}
\label{prop:largeN}
For every $N\in\N$, let $\G_N$ be the graph in Figure \ref{grafo1}. There exists $\overline{N}\in\N$ such that for every $N\geq\overline{N}$,
\[
s_\lambda=c_\lambda(\G_N)<\sigma_\lambda(\G_N)
\]
and neither $c_\lambda(\G_N)$ nor $\sigma_\lambda(\G_N)$ is attained.
\end{theorem}

Preliminary to the proof of Theorem \ref{prop:largeN}, we introduce some notation. Since $\lambda$ is fixed we omit to write it in quantities that depend on it, except for the levels $s_\lambda, c_\lambda(\G_N)$ and $\sigma_\lambda(\G_N)$ (and the same for quantities relative to $\widetilde\G_N$). Next, in order to have lighter notation, we omit to write the dependence from $N$ in various quantities, such as the loops $\LL_k$, and we keep it only in the names of the graphs. It is understood anyway that $N$ is a parameter that we will tune in various proofs.

In what follows, we split the set of solutions to problem \eqref{NLS},  as
\[
\Sf(\G_N) =\Sf_1(\G_N)\cup\Sf_2(\G_N)\cup\Sf_3(\G_N),
\]
where
\[
\begin{split}
&\Sf_1(\G_N):=\left\{u\in\Sf(\G_N)\,\mid \,\|u\|_{L^\infty(\G_N)}=\|u\|_{L^\infty(\mathcal{R}_k)},\,\text{for some }k\in\Z\right\}, \\ 
&\Sf_2(\G_N):=\left\{u\in\Sf(\G_N)\,\mid\,\|u\|_{L^\infty(\G_N)}=\|u\|_{L^\infty(e)},\,\text{for some edge }e \in\G_N\text{ of length $1$}\right\}, \\
&\Sf_3(\G_N):=\left\{u\in\Sf(\G_N)\,\mid\,\|u\|_{L^\infty(\G_N)}=\|u\|_{L^\infty(\LL_k)},\,\text{for some self-loop }\LL_k  \in\G_N \text{ of length $N$}\right\}.
\end{split}
\]
We define analogously the sets $\Sf_1(\widetilde{\G}_N)$, $\Sf_2(\widetilde{\G}_N)$, $\Sf_3(\widetilde{\G}_N)$ so that we also have
\begin{equation}
\label{tildesplit}
\Sf(\widetilde{\G}_N)=\Sf_1(\widetilde{\G}_N)\cup\Sf_2(\widetilde{\G}_N)\cup\Sf_3(\widetilde{\G}_N)\,.
\end{equation}

\begin{remark}
\label{rem:assH}
By construction, both $\G_N$ and $\widetilde{\G}_N$ satisfy assumption (H), so we have immediately that for every $N\in\N$,
\[
c_\lambda(\G_N)=c_\lambda(\widetilde{\G}_N)=s_\lambda
\]
and neither infimum is attained as a consequence of Theorem \ref{notatt}.
\end{remark}

\begin{remark}
\label{rem:pos}
Observe that, if $\G\in{\bf G}$ satisfies assumption (H) and $u$ is a sign-changing solution, then $u^+\in \mathcal{N}_{\lambda}(\G)$ and $u^-\in \mathcal{N}_{\lambda}(\G) $ and, by Theorem  \ref{notatt},
$$
J_{\lambda}(u)=J_{\lambda}(u^+)+ J_{\lambda}(u^-)\geq  2\inf_{u \in \NL(\G) } J_{\lambda}(u)=2s_{\lambda}.
$$
\end{remark}

The proof of Theorem \ref{prop:largeN} will rely on the next series of lemmas.
As $J_{\lambda}(u)=J_{\lambda}(|u|)$, we will assume without loss of generality in the rest of this section that all the functions $u \in \NN_\lambda(\G_N)$ that we will consider will be non-negative.

\begin{lemma}
\label{lemmaremark}
Let $e \in \G_N$, $N\ge 2$, be any bounded edge. Identify $\G_N\setminus \BB$ with $\Gt \setminus \widetilde\LL_0$ and let $\widetilde e \in \Gt$ be
\[
\widetilde e = \begin{cases} e & \text{ if } e \notin \BB \\ 
\widetilde\LL_0 & \text{ if } e \in \BB. \end{cases}
\]
Then for every $u \in \NN_\lambda(\G_N) \cap X_e$, there exists $\widetilde v \in   \NN_\lambda(\Gt) \cap X_{\widetilde e}$ such that
\begin{equation}
\label{ineq1}
J_{\lambda}(\widetilde v) \le J_{\lambda}(u).
\end{equation}
Moreover, if for $t$  in a set of  positive measure, 
\begin{equation}
\label{piu3}
\#\{x\in \BB\mid u(x)=t\} \ge 3,
\end{equation}
then the inequality \eqref{ineq1} is strict.
\end{lemma}
\begin{proof} 
Since $\BB$ contains a loop, 
\[
\#\{x\in \BB\mid u(x)=t\} \ge 2 \quad\text{ for every } t \in \big(\min_\BB u, \max_\BB u \big).
\]
Letting  $\widehat u (x)$ be the symmetric rearrangement of  the restriction of $u$ to $\BB$,   $\widehat u \in H^1(-N/2,N/2)$ and 
\begin{equation}
\label{rearr1}
\|\widehat u\|_{L^q(-N/2,N/2)} = \|u\|_{L^q(\BB)} \quad\forall q \in [1,+\infty], \qquad \|\widehat u'\|_{L^2(-N/2,N/2)} \le \|u'\|_{L^2(\BB)}.
\end{equation}
Furthermore, as $\widehat u (p) = u(\vv_0)$ for some $p$ in $[-N/2,N/2]$, we can view $\widehat u$ as a function on $\widetilde\LL_0$ (after identifying $\widetilde\LL_0$ with $(-N/2,N/2)$ and $p$ with $\vv_0$). We can then define $ v\in H^1(\Gt)$ as
\[
 v (x):=\begin{cases}
u(x) & \text{if }x\in\widetilde{\G}_N\setminus\widetilde \LL_0\ = \G_N \setminus \BB \\
\widehat u(x) & \text{if }x\in \widetilde\LL_0.
\end{cases}
\]
Continuity is guaranteed since $\widehat u(p) =  u(\vv_0)$.
By construction, $\| v '\|_{L^2(\widetilde{\G}_N)} \le \|u'\|_{L^2(\G_N)}$ and $\| v \|_{L^q(\widetilde{\G}_N)} = 
\|u\|_{L^q(\G_N)}$ for every $q \in[1+\infty]$, leading to $\pi_\lambda( v ) \le 1$.
	
Moreover, by construction, if $u$ attains its $L^\infty$ norm on some edge  $e \in \G_N\setminus \BB$, then $ v$ attains its $L^\infty$ norm on the corresponding edge of $\Gt \setminus \widetilde\LL_0$; if, instead,  $u$ attains its $L^\infty$ norm on some edge  of $\BB$, then  $ v$ attains it  on  $\widetilde \LL_0$. This shows that $\widetilde v:=\pi_\lambda( v) v \in\  \NN_\lambda(\Gt) \cap X_{\widetilde e}$ and
\[
J_{\lambda}(\widetilde v) = J_{\lambda}(\pi_\lambda( v) v)  =  \kappa \pi_\lambda( v)^p \|v\|_p^p\leq \kappa \|u\|_p^p=J_{\lambda}(u). 
\]
Finally, in case \eqref{piu3} holds, the inequality in \eqref{rearr1} is strict by Proposition \ref{Ncontr}, resulting in the strict inequality in \eqref{ineq1}.
\end{proof}

The next lemma shows that the action of solutions attaining their maximum on a half-line cannot be too close to $s_\lambda$.

\begin{lemma}
	\label{lem1}
	There exists $\delta_1>0$ such that for every $N\in\N$,
	\[
	\inf_{u\in\Sf_1(\G_N)}J_{\lambda}(u)\geq s_\lambda+\delta_1\qquad\text{and}\qquad\inf_{u\in\Sf_1(\widetilde{\G}_N)}J_{\lambda}(u)\geq s_\lambda+\delta_1\,.
	\]
\end{lemma}

\begin{proof}
We prove the statement explicitly for $\G_N$ only, as the argument works exactly in the same way for $\widetilde{\G}_N$. 
	
Let $u \in \Sf_1(\G_N)$. If $u$ is sign-changing  then, by Remark \ref{rem:pos}, we know that 
$J_{\lambda}(u)\ge 2 s_\lambda$.
Thus we consider the case where  $u$ does not change sign. We can suppose that $u\geq 0$ and by strong maximum principle, we have that $u>0$ on $\G$.

Call $\HH_1 \subset \RR_k$, for some $k\in \Z$, the half-line where $u$ attains its $L^\infty$ norm and $\HH_2$ the other half-line emanating from $\vv_k$. 
Suppose that $u$ attains its maximum at $\vv_k$. Then
\begin{equation}
\label{mag3}
\# u^{-1}(t)\geq 3\qquad\forall t\in(0,\|u\|_{\infty}),
\end{equation}
since every value $t\in(0,\|u\|_{\infty})$ is attained at least once in $\HH_1$, in $\HH_2$ and in $\G_N \setminus \RR_k$.

Then Proposition \ref{sK2} applied with $K=3$ yields
\begin{equation}
\label{liv1}
J_{\lambda}(u)\ge \frac32 s_\lambda.
\end{equation}
Suppose now that $u$ attains its maximum in the interior of $\HH_1$. Since $u$ solves \eqref{NLS},  its restrictions to  $\HH_1$ and $\HH_2$ coincide with suitable parts of the soliton $\phi_\lambda$. In particular, since $u$ attains its maximum in the interior of $\HH_1$, we see that 
$u(x) = \phi_\lambda(x-a)$, for some $a >0$, on $\HH_1$. Therefore, by continuity at $\vv_k$, on $\HH_2$ either $u(x) = \phi_\lambda(x-a)$ as well, or $u(x) = \phi_\lambda(x+a)$. 

In the first case we have a copy of $\phi_\lambda(x-a)$ on each of the two half-lines,  and the maximum of $u$ is attained on both half-lines. But then \eqref{mag3} again holds for $u$, since every value in $(u(\vv_k), \|u\|_\infty)$ is attained twice on each half-line,
and every value in $(0,u(\vv_k))$ is attained once on each half-line and at least once in $\G_N\setminus \RR_k$. Thus we conclude, exactly as above, that \eqref{liv1} holds.

In the second case the restriction of $u$ to $\RR_k$ is the whole soliton $\phi_\lambda$, which is smooth, and in particular the derivatives of $\phi_\lambda$ at $\vv_k$, being opposed, do not contribute to the Kirchhoff condition. Therefore $u$ solves problem \eqref{NLS} on $\G_N\setminus \RR_k$ and, as such, $u\in \NN_\lambda(\G_N\setminus \RR_k)$. Then, by \eqref{pbelow},
\[
\|u\|_{L^p\left(\G_N\setminus\RR_k\right)}^p   \ge C,
\]
with $C$ depending only on $\lambda$ and $p$.
In conclusion,
\[
J_{\lambda}(u) = \kappa \left( \|u\|_{L^p\left(\RR_k\right)}^p + 
\|u\|_{L^p\left(\G_N\setminus\RR_k\right)}^p    \right)
\\
\ge  \kappa \left(\|\phi_\lambda\|_{L^p\left(\mathbb R\right)}^p+  C\right) = s_\lambda + \kappa C, 
\]
and the lemma is proved choosing $\delta_1 = \min\left\{\frac12 s_\lambda, \kappa C\right\}$. 
\end{proof}

Now we prove that an estimate similar to that of the previous lemma holds also for solutions attaining their maximum on an edge of length $1$.

\begin{lemma}
\label{lem2}
There exists $\delta_2>0$ such that for every $N \ge 2$,
\begin{equation}
\label{2inf}
\inf_{u\in\Sf_2(\G_N)}J_{\lambda}(u)\geq s_\lambda+\delta_2\qquad\text{and}\qquad\inf_{u\in\Sf_2(\widetilde{\G}_N)}J_{\lambda}(u)\geq s_\lambda+\delta_2\,.
\end{equation}
\end{lemma}

\begin{proof}
We will prove the existence of $\delta_2>0$ such that for every $N \ge 2$,
\[
\inf_{u\in\mathcal{N}_2(\widetilde{\G}_N)}J_{\lambda}(u)\geq s_\lambda+\delta_2
\]
where 
\[
\mathcal{N}_2(\widetilde{\G}_N)=\left\{u\in\mathcal{N}(\widetilde{\G}_N)\,\mid\,\|u\|_{L^\infty(\widetilde{\G}_N)}=\|u\|_{L^\infty(e)},\,\text{for some edge }e \in\widetilde{\G}_N\text{ of length $1$}\right\}.
\]
 Then the result will follow via Lemma \ref{lemmaremark} and \eqref{csigma}.

Let $u \in \mathcal{N}_2(\Gt)$ and assume that $u$ attains its maximum on a given edge $e$ (of length $1$). 

We first construct a new graph and we rearrange $u$ on it. Let $\overline{e}$ be an edge of length $1$ and attach a pair of half-lines $\HH_1,\HH_2$ to one of its vertices and another pair  $\HH_3,\HH_4$ to the other one. We obtain in this way an H-shaped graph denoted by $\overline{\G}$. 
We now claim that there exists $v\in\NN_\lambda(\overline{\G})\cap X_{\overline{e}}$  such that 
\begin{equation}
\label{v<u}
J_{\lambda}(v)\leq J_{\lambda}(u).
\end{equation}
To see this, let $\vv_k$, $\vv_{k+1}$ be the vertices of $e$ and let $\widetilde\G_{\vv_k}, \widetilde\G_{\vv_{k+1}}$ be the connected components of $\Gt\setminus e$ containing $\vv_k$ and $\vv_{k+1}$ respectively. 
Note that both $\widetilde\G_{\vv_k}$ and $\widetilde\G_{\vv_{k+1}}$ satisfy assumption (H) and hence 
\begin{center}
$\#\{x\in\widetilde\G_{\vv_k} \mid u(x)=t\} \ge 2\qquad \text{for almost every } t \in (0,\|u\|_{L^{\infty}( \widetilde\G_{\vv_k})}),$
\\
$\#\{x\in\widetilde\G_{\vv_{k+1}} \mid u(x)=t\} \ge 2\qquad \text{for almost every } t \in (0,\|u\|_{L^{\infty}( \widetilde\G_{\vv_{k+1}})}).$
\end{center}
Let $u_1$ and $u_2$ be the symmetric rearrangements on $\R$ of the restrictions of $u$ to  
$\widetilde\G_{\vv_k}$, $\widetilde\G_{\vv_{k+1}}$ respectively, so that  $u_1, u_2 \in H^1(\R)$ and 
$$
\begin{array}{c}
\|u_1\|_{L^q(\R)}=\|u\|_{L^q(\widetilde\G_{\vv_k})}, \quad
\|u_2\|_{L^q(\R)}=\|u\|_{L^q(\widetilde\G_{\vv_{k+1}})} \mbox{ for every } q\in[1,\infty],
\\ 
\|u_1'\|_{L^2(\R)}\leq\|u'\|_{L^2(\widetilde\G_{\vv_k})}, \quad
\|u_2'\|_{L^2(\R)}\leq\|u'\|_{L^2(\widetilde\G_{\vv_{k+1}})}.
\end{array}
$$
 Furthermore, there exist $x_1,x_2\in\R$ such that 
$u_1(x_1)=u(\vv_k)$ and $u_2(x_2)=u(\vv_{k+1})$. Define then $\overline{u}\in H^1(\overline{\G})$ as
\[
\overline{u}(x)=\begin{cases}
u(x) & \text{ if }x\in \overline{e}\\
u_1(x+x_1) & \text{ if }x\in\HH_1\cup\HH_2\\
u_2(x+x_2) & \text{ if }x\in\HH_3\cup\HH_4,
\end{cases}
\]	
where with a slight abuse of notation we identified $\overline{e}\in \overline{\G}$ with $e\in\Gt$. By construction, $\overline{u}$ attains its $L^\infty$ norm on $\overline{e}$, $\|\overline{u}\|_{L^q(\overline{\G})}=\|u\|_{L^q(\widetilde\G_N)}$ for every $q$ and $\|\overline{u}'\|_{L^2(\overline{\G})}\leq\|u'\|_{L^2(\widetilde\G_N)}$, so that $\pi_\lambda(\overline{u})\leq1$. Hence, setting $v=\pi_\lambda(\overline{u})\overline{u}$, we obtain $v\in\NN_\lambda(\overline{\G})\cap X_{\overline{e}}$ fulfilling \eqref{v<u}. Thus
\[
J_{\lambda}(u) \ge \inf_{v\in\NN_\lambda(\overline{\G})\cap X_{\overline{e}}}J_{\lambda}(v) =c_\lambda(\overline{\G},\overline e)
\]
and it suffices to show that $c_\lambda(\overline{\G},\overline e)> s_\lambda$.

Now if $c_\lambda(\overline{\G},\overline e)$ is attained this is trivial by Theorem \ref{notatt}, since $\overline{\G}\in \bf{G}$ satisfies assumption (H).  If $c_\lambda(\overline{\G},\overline e)$ is not attained and $(w_n)_n \subset \NN_\lambda(\overline{\G})\cap X_{\overline{e}}$ is a minimizing sequence, then Lemma \ref{prova} applies yielding a weak limit $w \in X_{\overline{e}}\setminus\{0\}$, with $m := \|w\|_2^2 < \lim_n \|w_n\|_2^2 =: \mu$, $L(w)<\lambda$ and
\[
\lim_n L(w_n-w) = \lambda +\frac{m}{\mu-m}(\lambda- L(w)) \ge \lambda.
\]
In conclusion, as usual by the  Brezis--Lieb Lemma and Lemma \ref{Lem 3.5}
\[
c_\lambda(\overline{\G},\overline e) = \lim_n J_{\lambda}(w_n)   = \lim_n \kappa\|w_n\|_p^p= \kappa \lim_n \big(\|w_n-w\|_p^p + \|w\|_p^p\big) \ge  s_1\lambda^\alpha + \kappa\|w\|_p^p=
s_\lambda + \kappa\|w\|_p^p 
\]
and the proof is complete also when $c_\lambda(\overline{\G},\overline e)$ is not attained as $w\not=0$.
\end{proof}

In the next lemma, we study the properties of the graph $\Gt$.

\begin{lemma}
\label{lem3}
For every $\eps>0$ there exists $N_\eps:=N_\eps(\eps,\lambda,p)\in\N$ such that, for every $N\geq N_\eps$, 
\[
\inf_{u\in\Sf_{3}(\widetilde{\G}_N)}J_{\lambda}(u)<s_\lambda+\eps 
\]
and it is attained.  
\end{lemma}
\begin{proof} 
By Remark \ref{level}, for every $\eps>0$ there exists $N_\eps \in\N$ such that for every $N\geq N_\eps$
\[
c_\lambda(\Gt, \widetilde \LL_0)  < s_\lambda + \eps.
\]
By taking $N_\eps$ even larger, if necessary,   Theorem \ref{exlarge} guarantees that $c_\lambda(\Gt, \widetilde \LL_0)$ is attained by some $u$ and, by Theorem \ref{usol}, $u \in \Sf_{3}(\Gt)$.
If $v$ is any other element of $\Sf_{3}(\Gt)$, by the periodicity of $\Gt$ there exists $\widetilde v \in \NN_\lambda(\Gt) \cap X_{\widetilde\LL_0}$ (a translation of $v$) such that $J_{\lambda}(\widetilde v ) = J_{\lambda}(v)$. Therefore
\[
J_{\lambda}(v) = J_{\lambda}(\widetilde v ) \ge c_\lambda(\Gt, \widetilde \LL_0)  = J_{\lambda}(u),
\]
which shows that $\displaystyle\inf_{u\in\Sf_{3}(\widetilde{\G}_N)}J_{\lambda}(u)=c_\lambda(\Gt, \widetilde \LL_0)$ is attained (by $u$).
\end{proof}

\begin{corollary}
\label{tildeatt}
Let $\eps \le \frac12 \min\{\delta_1,\delta_2\}$, where $\delta_1$, $\delta_2$ are given by Lemmas \ref{lem1}--\ref{lem2}, and $N_\eps$ be the corresponding number given by Lemma \ref{lem3}. For every $N\geq N_\eps$, 
\[
\sigma_\lambda(\widetilde{\G}_N) = \inf_{u\in\Sf_3(\widetilde{\G}_N)}J_{\lambda}(u)<s_{\lambda}+\eps.
\]
and $\sigma_\lambda(\widetilde{\G}_N)$ is attained.
\end{corollary}

\begin{proof}
By Lemmas  \ref{lem1}--\ref{lem2}, for every $N\ge 2$,
\[
\inf_{u\in\Sf_1(\widetilde{\G}_N)}J_{\lambda}(u)\geq s_\lambda+\eps, \qquad \inf_{u\in\Sf_2(\widetilde{\G}_N)}J_{\lambda}(u)\geq s_\lambda+\eps,
\]
while by Lemma \ref{lem3}
\[
\inf_{u\in\Sf_3(\widetilde{\G}_N)}J_{\lambda}(u)<s_\lambda+\eps.
\]
Therefore, in view of \eqref{tildesplit},
\[
\sigma_\lambda(\widetilde{\G}_N)  = \inf_{u\in\Sf_3(\widetilde{\G}_N)}J_{\lambda}(u)
\]
and it is attained, again by Lemma \ref{lem3}.
\end{proof}

To proceed, we prove a further preliminary result, similar in spirit to Lemma \ref{lemmaremark}, that establishes a {\em strict} inequality when passing from $\Sf_3(\G_N)$ to $ \Sf_3(\widetilde\G_N)$.

\begin{lemma}
\label{exclaim}
For every $N$ large enough and  for every $u \in \Sf_3(\G_N)$ such that $J_{\lambda}(u) < 2s_\lambda$, there exists $ v \in  \Sf_3(\widetilde\G_N)$ such that $J_{\lambda}( v) < J_{\lambda}(u)$.
\end{lemma}

\begin{proof} Let $u \in \Sf_3(\G_N)$ such that $J_{\lambda}(u) < 2s_\lambda$. By Remark \ref{rem:pos},   $u$ does not change sign. We suppose that $u\geq 0$ and by strong maximum principle, we have that $u>0$ on $\G$.

If $u$ attains its maximum in a loop $\LL_k $, by Lemma \ref{lemmaremark} there exists $\widetilde v \in \NN_\lambda(\Gt) \cap X_{\widetilde\LL_k}$   such that $J_{\lambda}(\widetilde v) \le J_{\lambda}(u)$, where $\widetilde\LL_k$ corresponds to $\LL_k$ after the identification of $\G_N\setminus \BB$ with $\Gt \setminus \widetilde\LL_0$.

If $u$ is  constant  on $m$ edges of $\BB$, then it necessarily equals $\lambda^{\frac1{p-2}}$, and 
\[
2s_\lambda\ge J_{\lambda}(u)  \ge m \kappa\lambda^{\frac{p}{p-2}} 
\]
shows that $m$ is bounded by a constant depending only on $\lambda$  and $p$.
Since $N$ can be assumed as large as we wish, there are at least $N-m \ge 6$ edges of $\BB$ on which $u$ is not constant. Any pair of these edges in $\BB$ forms a loop on which $u$ is not constant. Since there are at least $3$ such loops, necessarily
\[
\#\{x\in \BB\mid u(x)=t\} \ge 3
\]
for $t$ in a set of positive measure, so that by Lemma \ref{lemmaremark}, the element $\widetilde v \in   \NN_\lambda(\Gt) \cap X_{\widetilde \LL_k}$ found above satisfies the strict inequality  $J_{\lambda}(\widetilde v) < J_{\lambda}(u)$. 
Finally, invoking again Theorems \ref{exlarge} and \ref{usol}, provided $N$ is sufficiently large, there exists $ v \in\Sf_3(\widetilde{\G}_N)$ such that $J_{\lambda}(v) = \inf_{w \in   \NN_\lambda(\Gt) \cap X_{\widetilde\LL_k}}J_{\lambda}(w)$ and hence
\[
J_{\lambda}(v) \le J_{\lambda}(\widetilde v) < J_{\lambda}(u),
\]
and the proof is complete.
\end{proof}

\begin{lemma}
\label{lem4}
There exists $\overline{N}\in\N$ such that, for every $N\geq\overline{N}$, we have $\sigma_\lambda(\widetilde{\G}_N)>s_\lambda$,
\begin{equation}
\label{leveq}
\sigma_\lambda(\G_N)=\sigma_\lambda(\widetilde{\G}_N)
\end{equation}
and $\sigma_\lambda(\G_N)$ is not attained.
\end{lemma}
\begin{proof}
Let $\eps$ and $N_\eps$ be as in Corollary \ref{tildeatt}, so that there exists $w\in\Sf_3(\widetilde{\G}_N)$ satisfying
\[
J_{\lambda}(w) =\sigma_\lambda(\widetilde{\G}_N)< s_\lambda + \eps.
\]
Note that  $\sigma_\lambda(\widetilde{\G}_N) = J_{\lambda}(w) > s_\lambda$ by Remark \ref{rem:assH}. We assume, in accordance with Theorem \ref{usol}, that $w >0$, and that $w$ attains its maximum in the loop $\widetilde\LL_0$, which is possible, as usual, by the periodicity of $\widetilde{\G}_N$. For every $\delta >0$, define $w_\delta \in H^1(\widetilde{\G}_N)$ as
\[
w_\delta(x) = (w(x)- \delta)^+
\]
and notice that $w_\delta \to w$ strongly in $H^1(\widetilde{\G}_N)$ as $\delta \to 0$. The support of $w_\delta$ is, by construction, a bounded subgraph $\G_\delta$ of $\widetilde{\G}_N$, that can also be considered as a subgraph of $\G_N$: it suffices to embed it, for every $\delta$, into $\G_N$ in such a way that it does not contain $\vv_0$. Thus, after extending $w_\delta$ to $0$ in $\G_N \setminus \G_\delta$, we can view it as a function in $H^1(\G_N)$ attaining its maximum on some loop $\LL_{k_\delta}$. Note also that, by strong convergence,
\[
\pi_\lambda(w_\delta)^{p-2} = \frac{\|w_\delta'\|_{L^2(\G_N)}^2 + \lambda \|w_\delta\|_{L^2(\G_N)}^2} {\|w_\delta\|_{L^p(\G_N)}^p }
= \frac{\|w_\delta'\|_{L^2(\widetilde{\G}_N)}^2 + \lambda \|w_\delta\|_{L^2(\widetilde{\G}_N)}^2} {\|w_\delta\|_{L^p(\widetilde{\G}_N)}^p } \to 1
\]
as $\delta \to 0$. 

By Theorems \ref{exlarge}--\ref{usol}, for every $\delta>0$ there exists $v_\delta\in\Sf(\G_N)$ such that 
\[
J_{\lambda}(v_\delta)=\inf_{u\in\NN_\lambda(\G_N)\cap X_{\LL_{k_{\delta}}}}J_{\lambda}(u).
\]
Therefore, as $\delta \to 0$,
\[
\sigma_\lambda(\G_N)\leq J_{\lambda}(v_\delta)\leq J_{\lambda}(\pi_\lambda(w_\delta)w_\delta) \to J_{\lambda}(w)  = \sigma_\lambda(\widetilde{\G}_N),
\]
showing that 
\begin{equation}
\label{dissigma}
\sigma_\lambda(\G_N)\leq\sigma_\lambda(\widetilde{\G}_N),
\end{equation} 
and hence,  by Lemmas \ref{lem1}--\ref{lem2}, 
\[
\sigma_\lambda(\G_N)=\inf_{u\in\Sf_3({\G}_N)}J_{\lambda}(u)<s_\lambda+\eps.
\]
Without loss of generality, let us assume that $\eps < s_\lambda$.
Next, for every $u \in \Sf_3(\G_N)$  with $J_{\lambda}(u) < s_\lambda+\eps$, let 
$v \in \Sf_3(\widetilde\G_N)$ be the function provided by Lemma \ref{exclaim}. Then
\[
\sigma_\lambda(\widetilde \G_N) \le J_{\lambda}(v) < J_{\lambda}(u)
\]
and taking the infimum over $u$ we obtain $\sigma_\lambda(\widetilde \G_N) \le \sigma_\lambda(\G_N)$ which, coupled with \eqref{dissigma}, establishes \eqref{leveq}. 

Finally, to prove that $\sigma_\lambda(\G_N)$ is not attained, assume instead that there exists $u \in \Sf_3(\G_N)$ such that
$J (u) = \sigma_\lambda(\G_N)$. By  Lemma \ref{exclaim} again, let $v \in  \Sf_3(\widetilde\G_N)$ satisfy $ J_{\lambda}(v) < J_{\lambda}(u)$. Then
\[
\sigma_\lambda(\widetilde \G_N) \le J_{\lambda}(v) < J_{\lambda}(u) = \sigma_\lambda(\G_N),
\]
contradicting \eqref{leveq}.
	\end{proof}
\begin{proof}[Proof of Theorem \ref{prop:largeN}]
	It is enough to take $\overline{N}$ as in Lemma \ref{lem4}, the result is then a straightforward consequence of Remark \ref{rem:assH} and Lemma \ref{lem4}.
\end{proof}

\section*{Acknowledgements}
S.D. acknowledges that this work has been partially supported by the INdAM GNAMPA project 2022 {\em Modelli matematici con singolarit\'a per fenomeni di interazione}.
\medbreak

\noindent
D.G. is an F.R.S.-FNRS Research Fellow.


\begin{thebibliography}{abcd}
	
	\bibitem{ABD}
	Adami R., Boni F., Dovetta S., {\em Competing nonlinearities in NLS equations as source of threshold phenomena on star graphs}, J. Funct. Anal. {\bf283}(1) (2022), 109483.
	
	
	\bibitem{ABR}
	Adami R., Boni F., Ruighi A., {\em Non--Kirchhoff vertices and nonlinear Schr\"odinger ground states on graphs}, Mathematics {\bf8}(4) (2020), 617.
	
	\bibitem{ACFN}
	Adami R., Cacciapuoti C., Finco D., Noja D., {\em Stable standing waves for a NLS on star graphs as local minimizers of the constrained energy}, J. Diff. Eq. {\bf260} (2016), 7397--7415.
	
	\bibitem{ACFN1}
	Adami R., Cacciapuoti C., Finco D., Noja D., {\em Stationary states of NLS on star graphs}, Europhysics Letters {\bf100}(1) (2012), 10003.
	
	
	\bibitem{ASTbound}
	Adami R., Serra E., Tilli P., 
	\textit{Multiple positive bound states for the subcritical NLS equation on metric graphs}, 
	Calc. Var. PDE {\bf 58}(5) (2019), 16pp.

	\bibitem{AST3}
	Adami R., Serra E., Tilli P., 
	{\em Negative energy ground states for the $L^2$--critical NLSE on metric graphs},
	Comm. Math. Phys. {\bf 352}(1) (2017), 387--406.
	
	\bibitem{AST1}
	Adami R., Serra E., Tilli P., \emph{NLS ground states on graphs}, Calc. Var. PDE \textbf{54} (2015), 743--761.

	
	\bibitem{AST2}
	Adami R., Serra E., Tilli P., \textit{Threshold phenomena and existence results for NLS ground states on graphs}, J. Funct. Anal. \textbf{271}(1) (2016), 201--223.	
	
	\bibitem{AP}
	Akduman S., Pankov A., {\em Nonlinear Schr\"odinger equation with growing potential on infinite metric graphs}, Nonlin. Anal. {\bf184} (2019), 258--272.
	
	\bibitem{BMP}
	Berkolaiko G., Marzuola J.L., Pelinovsky D.E., \emph{Edge--localized states on quantum graphs in the limit of large mass}, Ann. Inst. H. Poincaré (C) An. Non Lin. {\bf 38}(5) (2021), 1295--1335.
	
	\bibitem{BDL20}
	Besse C., Duboscq R., Le Coz S., {\em Gradient flow approach to the calculation of ground states on nonlinear quantum graphs}, Ann. H. Lebesgue {\bf5} (2022), 387--428.
	
	\bibitem{BDL21}
	Besse C., Duboscq R., Le Coz S., {\em Numerical simulations on nonlinear quantum graphs with the GraFiDi library}, SMAI J. Comp. Math. {\bf8} (2022), 1--47.
	
	\bibitem{BD22}
	Boni F., Dovetta S., \emph{Doubly nonlinear Schr\"odinger ground states on metric graphs}, Nonlinearity {\bf 35} (2022), 3283--3323.
	
	\bibitem{BD21}
	Boni F., Dovetta S., \emph{Ground states for a doubly nonlinear Schr\"odinger equation in dimension one}, J. Math. Anal. Appl. {\bf496}(1) (2021), 124797.
	
\bibitem{BCJS}
	Borthwick J., Chang X., Jeanjean L., Soave N., {\em Bounded Palais-Smale sequences with Morse type information for some constrained functionals},  	arXiv:2210.12626 [math.AP] (2022). 
	
	\bibitem{BCJS2}
	Borthwick J., Chang X., Jeanjean L., Soave N., {\em Normalized solutions of $L^2$-supercritical NLS equations on noncompact metric graphs with localized nonlinearities},  	arXiv:2212.04840 [math.AP] (2022). 
	
	\bibitem{BL}
	Brezis H., Lieb E., \textit{A relation between pointwise convergence of functions and convergence of functionals}, Proc. Amer. Math. Soc. \textbf{88}(3) (1983), 486--490.
	
	\bibitem{CFN}
	Cacciapuoti C., Finco D., Noja D., {\em Ground state and orbital stability for the NLS equation on a general starlike graph with potentials}, Nonlinearity {\bf30} (2017), 3271--3303.
	
	\bibitem{cazenave}
Cazenave T., Semilinear Schr\"odinger Equations, Courant Lecture Notes 10. American
Mathematical Society, Providence, RI, 2003.

	
	\bibitem{CL}
	Cazenave T., Lions P.-L., {\em Orbital stability of standing waves for some nonlinear Schr\"odinger equations}, Comm. Math. Phys. {\bf85} (1982), 549--561.
	
	\bibitem{CJS}
	Chang X., Jeanjean L., Soave N., {\em Normalized solutions of $L^2$--supercritical NLS equations on compact metric graphs}, arXiv:2204.01043 [math.AP] (2022). 
	
	\bibitem{D}
	Dovetta S., {\em Existence of infinitely many stationary solutions of the $L^2$--subcritical and critical NLSE on compact metric graphs}, J. Differential Equations \textbf{264}(7) (2018), 4806--4821.
	
	\bibitem{DGMP}
	Dovetta S., Ghimenti M., Micheletti A.M., Pistoia A., {\em Peaked and low action solutions of NLS equations on graphs with terminal edges}, SIAM J. Math. Anal. {\bf52}(3) (2020), 2874--2894.
	
	\bibitem{DSTma}
	Dovetta S., Serra E., Tilli P., {\em Action versus energy ground states in nonlinear Schr\"odinger equations}, Math. Ann. (2022) https://doi.org/10.1007/s00208-022-02382-z.
	
	\bibitem{DT19}
	Dovetta S., Tentarelli L., {\em $L^2$--critical NLS on noncompact metric graphs with localized nonlinearity: topological and metric features}, Calc. Var. PDE {\bf58}(3) (2019), art. n. 108.
	
	\bibitem{DT22}
	Dovetta S., Tentarelli L., {\em Symmetry breaking in two--dimensional square grids: persistence and failure of the dimensional crossover}. J. Math. Pures Appl. {\bf160} (2022), 99--157.
	
	\bibitem{Duff}
	Duff G. F. D.,
	{\em Integral inequalities for equimeasurable rearrangements}.
	Canad. J. Math.  {\bf 22} (1970), 408--430.

	\bibitem{F}
	Friedlander L.,
	{\em Extremal properties of eigenvalues for a metric graph}.
	Ann. Inst. Fourier {\bf 55}(1) (2005), 199--211.
	
	\bibitem{G}
	Goloshchapova N., {\em Dynamical and variational properties of the NLS--$\delta_s'$ equation on the star graph}, J. Diff. Eq. {\bf310} (2022), 1--44.
	
	\bibitem{J}
	Jeanjean L., {\em Existence of solutions with prescribed norm for semilinear elliptic equations}, Nonlin. Anal. {\bf28}(10) (1997), 1633--1659.
	
	\bibitem{JL}
	Jeanjean L., Lu S.-S, {\em On global minimizers for a mass constrained problem}, arXiv:2108.04142 [math.AP] (2021).
	
	\bibitem{GKP}
	Kairzhan A., Pelinovsky D.E., Goodman R.H., {\em Drift of spectrally stable shifted states on star graphs}, SIAM J. Appl. Dynam. Syst. {\bf18} (2019), 1723--1755.
	
	\bibitem{KMPX}
	Kairzhan A., Marangell R., Pelinovsky D.E., Xiao K., {\em Existence of standing waves on a flower graph}, J. Diff. Eq. {\bf271} (2021), 719--763.
	
	\bibitem{KNP}
	Kairzhan A., Noja D., Pelinovsky D.E., {\em Standing waves on quantum graphs}, J. Phys. A {\bf55}(24) (2022), 243001.
	
	\bibitem{KS}
	Kurata K., Shibata M., {\em Least energy solutions to semi--linear elliptic problems on metric graphs}, J. Math. Anal. Appl. {\bf491} (2020), 124297.

\bibitem{LC}
	Le Coz S., {\em Standing waves in nonlinear Schr\"odinger equations}, Analytical and Numerical Aspects
of Partial Differential Equations, de Gruyter, pp.151-192, 2008.

	\bibitem{NPS}
	Noja D.,  Pelinovsky D.E., Shaikhova G., {\em Bifurcations and stability of standing waves in the nonlinear Schr\"odinger equation on the tadpole graph}, 
	Nonlinearity {\bf28} (2015),  2343--2378.
	
	\bibitem{NP}
	Noja D.,  Pelinovsky D.E.,
	\emph{Standing waves of the quintic NLS equation on the tadpole graph}, Calc. Var. PDE {\bf59}(5) (2020), art. n. 173.
	
	\bibitem{PS22}
	Pierotti D., Soave N., 
	{\em Ground states for the NLS equation with combined nonlinearities on non-compact metric graphs}, SIAM J. Math. Anal. {\bf54}(1) (2022), 768--790.
	
	\bibitem{PSV}
	Pierotti D., Soave N., Verzini G.,
	\emph{Local minimizers in absence of ground states for the critical NLS energy on metric graphs}, Proc. Royal Soc. Edinb. Section A: Math. {\bf151}(2) (2021), 705--733.
	
	\bibitem{T}
	Tentarelli L., {\em NLS ground states on metric graphs with localized nonlinearities},
	J. Math. Anal. Appl. {\bf 433}(1) (2016), 291--304.
	
	
\end{thebibliography}
\end{document}